\documentclass[a4paper,reqno,10pt]{amsart}

\usepackage{epsf,graphicx,epsfig}
\usepackage{amscd}
\usepackage{amsmath,latexsym,amssymb,amsthm}
\usepackage[nospace,noadjust]{cite}
\usepackage{textcomp}
\usepackage{amsfonts}
\usepackage{setspace,cite}
\usepackage{lscape,fancyhdr,fancybox}
\usepackage{stmaryrd}
\usepackage{mathrsfs}
\usepackage[all,cmtip]{xy}
\usepackage{tikz}
\usepackage{cancel}
\usetikzlibrary{shapes,arrows,decorations.markings}
\usepackage{graphicx}
\usepackage{tikz-cd}
\usepackage{color}
\usepackage{url}
\usepackage{enumerate}
\usepackage[mathscr]{euscript}
  \theoremstyle{plain}

\swapnumbers
    \newtheorem{thm}{Theorem}[section]
    \newtheorem{prop}[thm]{Proposition}

    \newtheorem{subsec}[thm]{}
\theoremstyle{definition}
    \newtheorem{defn}[thm]{Definition}
        \newtheorem{remark}[thm]{Remark}
    \newtheorem{exam}[thm]{Example}

\theoremstyle{remark}

\title{}
\author{}
\date{}
\usepackage{amssymb}

\usepackage{hyperref}
\hypersetup{
	colorlinks,
	citecolor=blue,
	filecolor=black,
	linkcolor=blue,
	urlcolor=black
}

\begin{document}

\title[]{Non-abelian cohomology of Nijenhuis Lie algebras and the inducibility of automorphisms and derivations}

\author{Apurba Das}
\address{Department of Mathematics,
Indian Institute of Technology, Kharagpur 721302, West Bengal, India}
\email{apurbadas348@gmail.com, apurbadas348@maths.iitkgp.ac.in}

%\author{Suman Majhi}
%\address{Department of Mathematics,
%Indian Institute of Technology, Kharagpur 721302, West Bengal, India}
%\email{majhisuman693@gmail.com}

%\author{Ramkrishna Mandal}
%\address{Department of Mathematics, Indian Institute of Technology, Kharagpur 721302, West Bengal, India}
%\email{ramkrishnamandal430@gmail.com}

\begin{abstract}
In this paper, we first introduce the non-abelian cohomology group of a Nijenhuis Lie algebra with values in another Nijenhuis Lie algebra and show that it parametrizes the isomorphism classes of all non-abelian extensions. In particular, we obtain a classification result for abelian extensions of a Nijenhuis Lie algebra by a given Nijenhuis representation. Next, given a non-abelian extension of Nijenhuis Lie algebras, we investigate the inducibility of a pair of Nijenhuis Lie algebra automorphisms and show that the corresponding obstruction lies in the non-abelian cohomology group. Subsequently, we also consider the inducibility of a pair of Nijenhuis Lie algebra derivations in a given abelian extension.
\end{abstract}

\maketitle

%\curraddr{}
%\email{}

%\subjclass[2010]{}
%\keywords{}

\medskip

\begin{center}
\noindent {2020 MSC classification:} 17B40, 17B55, 17B56.

\noindent  {Keywords:} Nijenhuis Lie algebras, Non-abelian cohomology, Automorphisms, Derivations, Wells exact sequence.
\end{center}

 %Quasi-twilled associative algebras, Deformation maps, Crossed homomorphisms, Rota-Baxter operators, Controlling algebras, Cohomology.
 
 %Rota-Baxter operators, twisted Rota-Baxter operators, Crossed homomorphism, Averaging operators, Reynolds operators, Cohomology, Deformation.

 %\medskip

%\noindent {\sf Date of resubmission:} July 26, 2021.

\thispagestyle{empty}

\tableofcontents

%\vspace{0.2cm}

\medskip

\section{Introduction}
Algebras are often equipped with distinguished operators. Among others, algebra homomorphisms, derivations and involutions are the most popular. In the last few years, Rota-Baxter operators and some variations have attracted much interest because of their importance in mathematics and mathematical physics. A Lie algebra endowed with a distinguished Rota-Baxter operator is referred to as a Rota-Baxter Lie algebra in the literature (see, for example, \cite{jiang-sheng}, \cite{das-hazra-mishra}, \cite{peng-zhang} and the references therein). Another interesting operator studied in various areas of mathematics (especially in algebra and geometry) and mathematical physics is the `Nijenhuis operator'. It is important to note that Nijenhuis operators significantly appear in the linear deformation theory of Lie algebras, complex Lie algebras, geometry of vector-valued differential forms and complex manifolds, bi-Hamiltonian systems, nonlinear evolution equations, integrable systems and tensor hierarchies in mathematical physics \cite{dorfman}, \cite{frol}, \cite{gra-bi}, \cite{koss}. A Lie algebra together with a distinguished Nijenhuis operator is called a {\em Nijenhuis Lie algebra}. It is well-known that a Nijenhuis operator on a Lie algebra gives rise to a new Lie algebra structure on the underlying space, called the deformed Lie algebra \cite{koss}. It has been shown in \cite{das-twisted} that Nijenhuis operators on Lie algebras are closely related to twisted Rota-Baxter operators and NS-Lie algebras (a structure that generalizes pre-Lie algebras). However, unlike Rota-Baxter operators, Nijenhuis operators cannot be characterized by their graphs as subalgebras of some bigger algebra. See also the references \cite{azimi}, \cite{lei}, \cite{leroux}, \cite{liu-sheng}, \cite{ma}, \cite{saha}, \cite{wang}, \cite{yuan} for some study of Nijenhuis operators on other types of algebras.

\medskip

Algebraic structures and their properties are often understood from some invariants associated with them. Each invariants have their significance. The notion of non-abelian cohomology of an abstract group with values in another group was introduced by Eilenberg and Maclane \cite{eilen} to understand non-abelian extensions of abstract groups (see also \cite{hochschild}). Subsequently, the non-abelian cohomology theory and non-abelian extensions were generalized to the context of associative algebras, Lie algebras, Leibniz algebras and various other algebraic structures (see, for example, \cite{ded}, \cite{lue}, \cite{gouray}, \cite{hoch}, \cite{inas}, \cite{fre}, \cite{casas}, \cite{liu-sheng-wang}, \cite{das-rathee}). Recently, the non-abelian cohomology of a Rota-Baxter Lie algebra with values in another Rota-Baxter Lie algebra is defined in \cite{das-hazra-mishra}. Our primary aim in this paper is to describe the {\em non-abelian cohomology group} of a Nijenhuis Lie algebra with values in another Nijenhuis Lie algebra. We show that the non-abelian cohomology group parametrizes the set of all isomorphism classes non-abelian extensions of Nijenhuis Lie algebras (cf. Theorem \ref{main-thm-sec3}). We also find a morphism from the non-abelian cohomology group of Nijenhuis Lie algebra to the non-abelian cohomology group of the deformed Lie algebra (cf. Remark \ref{remark-nl-def}). As a particular case, we define the {\em second cohomology group} of a Nijenhuis Lie algebra with coefficients in a Nijenhuis representation and obtain its correspondence with abelian extensions (cf. Theorem \ref{thm-abelian}). In subsequent work, we aim to develop the full cochain complex and the cohomology of a Nijenhuis Lie algebra with coefficients in a Nijenhuis representation. We also look forward to applications in deformation theory and homotopy algebras.

\medskip

A well-known question involving extensions of algebraic structures is the inducibility problem of a pair of automorphisms. This problem was first proposed by Wells \cite{wells} in the context of abstract groups. In the same article, he also obtained a short exact sequence connecting various automorphism groups (popularly known as the {\em Wells exact sequence}). His study was further explored in various specific cases and generalized to other algebraic structures \cite{bar-singh}, \cite{das-rathee}, \cite{hazra-habib}, \cite{hou-zhao}, \cite{jin}. In the context of Lie algebras, the inducibility problem can be stated as follows. Let $0 \rightarrow \mathfrak{h} \xrightarrow{i} \mathfrak{e} \xrightarrow{p} \mathfrak{g} \rightarrow 0$ be a given non-abelian extension of Lie algebras. Consider the group $\mathrm{Aut}_\mathfrak{h} (\mathfrak{e})$ of all Lie algebra automorphisms $\gamma \in \mathrm{Aut} (\mathfrak{e})$ for which $\mathfrak{h}$ is an invariant subspace. Then there is a group homomorphism $\tau : \mathrm{Aut}_\mathfrak{h} (\mathfrak{e}) \rightarrow \mathrm{Aut} (\mathfrak{h}) \times \mathrm{Aut} (\mathfrak{g})$ given by $\tau (\gamma) = (\gamma |_\mathfrak{h}, p \gamma s)$, where $s$ is any section of the map $p$. A pair $(\beta, \alpha) \in \mathrm{Aut} (\mathfrak{h}) \times \mathrm{Aut} (\mathfrak{g})$ is said to be inducible if it lies in the image of the map $\tau$. The inducibility problem then asks to find a necessary and sufficient condition under which a pair of Lie algebra automorphisms $(\beta, \alpha) \in \mathrm{Aut} (\mathfrak{h}) \times \mathrm{Aut} (\mathfrak{g})$ is inducible. To answer this problem, one formulates a suitable Wells map and obtains a necessary and sufficient condition for the inducibility problem in terms of the Wells map. In the Lie algebra case also, the Wells map fits into a short exact sequence connecting various automorphism groups. In \cite{das-hazra-mishra}, together with Hazra and Mishra, the present author considered the extensions of Rota-Baxter Lie algebras and studied the inducibility problem of a pair of Rota-Baxter Lie algebra automorphisms. Among others, they defined the Wells map in the context of Rota-Baxter Lie algebras and showed that a pair of Rota-Baxter Lie algebra automorphisms is inducible if and only if its image under the Wells map vanishes identically. Therefore, it is natural to ask whether the inducibility problem can be answered in the context of Nijenhuis Lie algebras. We aim to answer this question by introducing a suitable Wells map in the context of Nijenhuis Lie algebras (cf. Theorem \ref{thm-ind-aut}). We also construct the corresponding version of the Wells exact sequence (cf. Theorem \ref{thm-wells-ses-aut}). Finally, we provide our particular attention to the inducibility problem in a given abelian extension.

\medskip

Another similar question that arises in the study of extensions is the inducibility problem of a pair of derivations \cite{barati}, \cite{tan-xu}, \cite{hou-zhao}. Note that this problem makes sense only in a given abelian extension. In the context of Lie algebras, this can be stated as follows. Let $0 \rightarrow V \xrightarrow{i} \mathfrak{e}\xrightarrow{p} \mathfrak{g} \rightarrow 0$ be a given abelian extension of the Lie algebra $\mathfrak{g}$ by a representation $V$. Consider the space $\mathrm{Der}_V (\mathfrak{e})$ of all Lie algebra derivations $D \in \mathrm{Der} (\mathfrak{e})$ for which the subspace $V$ is invariant. Then there is a map $\eta : \mathrm{Der}_V (\mathfrak{e}) \rightarrow \mathrm{Der} (V) \times \mathrm{Der} (\mathfrak{g})$, $\eta (D) = (D \big|_V , pDs)$, where $s$ is any section of the map $p$. Here the inducibility problem also asks to find a necessary and sufficient condition for which a pair of Lie algebra derivations $(D_V, D_\mathfrak{g}) \in \mathrm{Der} (V) \times \mathrm{Der} (\mathfrak{g})$ lies in the image of $\eta$. As before, one usually formulates a suitable Wells map and obtains the corresponding obstruction as the image of the Wells map. The Wells map here also fits into a short exact sequence that connects various derivation spaces. Our next aim in this paper is to generalize this inducibility problem in the context of Nijenhuis Lie algebras. For this, we first consider the notion of Nijenhuis Lie algebra derivations. Then given an abelian extension of a Nijenhuis Lie algebra by a fixed Nijenhuis representation, we define the space of {\em compatible pairs} of Nijenhuis Lie algebra derivations which turns out to be a Lie algebra. We also consider the Wells map in this context and show that a pair of Nijenhuis Lie algebra derivations is inducible if and only if the pair is compatible and its image under the Wells map vanishes (cf. Theorem \ref{thm-ind-der}). Finally, we end this part by showing that the Wells map here also fits a short exact sequence (cf. Theorem \ref{thm-wells-ses-der}).

\medskip

The paper is organized as follows. In Section \ref{sec2}, we recall some basics of Nijenhuis Lie algebras including their representations. The non-abelian cohomology group of a Nijenhuis Lie algebra with values in another Nijenhuis Lie algebra is defined in Section \ref{sec3}. We also show that the set of all isomorphism classes of non-abelian extensions of Nijenhuis Lie algebras can be parametrized by the non-abelian cohomology group. In Section \ref{sec4}, we consider the inducibility problem of a pair of Nijenhuis Lie algebra automorphisms in a given non-abelian extension. We also define the Wells map and the Wells exact sequence connecting various automorphism groups. Finally, in Section \ref{sec5}, we discuss the inducibility of a pair of Nijenhuis Lie algebra derivations in a given abelian extension and derive the corresponding Wells map and the Wells exact sequence.

\medskip

\section{Nijenhuis Lie algebras}\label{sec2}
In this section, we highlight some background on Nijenhuis operators and Nijenhuis Lie algebras. Among others, we consider Nijenhuis representations of Nijenhuis Lie algebras.

\begin{defn}
    Let $(\mathfrak{g}, [~, ~]_\mathfrak{g})$ be a Lie algebra. A {\bf Nijenhuis operator} on this Lie algebra is a linear map $N : \mathfrak{g} \rightarrow \mathfrak{g}$ that satisfies
    \begin{align*}
        [N (x), N (y)]_\mathfrak{g} = N ([ N (x), y]_\mathfrak{g} + [x, N(y)]_\mathfrak{g} - N [x, y]_\mathfrak{g} ), \text{ for } x, y \in \mathfrak{g}. 
    \end{align*}
\end{defn}

%Nijenhuis operators play a fundamental role in deformation theory, complex Lie algebras, quantum bi-Hamiltonian systems, integrable systems and tensor hierarchies in mathematical physics. 

Here we list a few examples of Nijenhuis operators on Lie algebras.

\begin{exam}\label{nij-op-exam}
(i) Let $(\mathfrak{g}, [~,~]_\mathfrak{g})$ be any Lie algebra. Then the identity map $\mathrm{Id}_\mathfrak{g}$ is trivially a Nijenhuis operator.

\medskip

(ii) Let $(\mathfrak{g}, [~,~]_\mathfrak{g})$ be any Lie algebra and $N: \mathfrak{g} \rightarrow \mathfrak{g}$ be a Nijenhuis operator on it.

\quad (a) Then for any scalar $\lambda \in {\bf k}$, the scalar multiplication $\lambda N$ is also a Nijenhuis operator.

\quad (b) For any $k \geq 0$, the $k$-th power $N^k$ is also a Nijenhuis operator. Moreover, for $k, l \geq 0$, the Nijenhuis operators $N^k$ and $N^l$ are compatible in the sense that their sum $N^k + N^l$ is also a Nijenhuis operator. In general, any polynomial $p (N)$ of the Nijenhuis operator $N$ is also a Nijenhuis operator \cite{koss}.

\medskip

(iii) Let $(A, ~ \! \cdot ~ \!)$ be an associative algebra. Then a linear map $N: A \rightarrow A$ is said to be a Nijenhuis operator \cite{gra-bi} on the associative algebra if 
\begin{align*}
N (a) \cdot N (b) = N ( N(a) \cdot b + a \cdot N (b) - N (a \cdot b)), \text{ for } a, b \in A.
\end{align*}
In this case, it is easy to see that $N$ is a Nijenhuis operator on the commutator Lie algebra $(A, [~,~])$, where $[a, b] = a \cdot b - b \cdot a$, for $a, b \in A$.

\medskip

(iv) A complex Lie algebra is a Lie algebra over the field of complex numbers. Equivalently, a complex Lie algebra can be described as a real Lie algebra $(\mathfrak{g}, [~,~]_\mathfrak{g})$ endowed with a $\mathbb{R}$-linear map $j : \mathfrak{g} \rightarrow \mathfrak{g}$ satisfying $j^2 = -\mathrm{id}_\mathfrak{g}$ and $[ j(x) , y]_\mathfrak{g} = [x, j (y)]_\mathfrak{g}  = j [x, y]_\mathfrak{g}$, for all $x, y \in \mathfrak{g}$. Then it turns out that $j$ is a Nijenhuis operator on the real Lie algebra  $(\mathfrak{g}, [~,~]_\mathfrak{g})$.

\medskip

(v) Let $(\mathfrak{g}, [~,~]_\mathfrak{g})$ be a Lie algebra whose underlying space $\mathfrak{g}$ has a direct sum decomposition $\mathfrak{g} = \mathfrak{g}_1 \oplus \mathfrak{g}_2$ into two Lie subalgebras. Let $p_1 , p_2 : \mathfrak{g} \rightarrow \mathfrak{g}$ be the projections onto the subspaces $\mathfrak{g}_1$ and $\mathfrak{g}_2$, respectively. Then $p_1$ and $p_2$ are both Nijenhuis operators. Any linear combination of $p_1$ and $p_2$ is also a Nijenhuis operator.

\medskip

(vi) Let $(\mathfrak{g}, [~,~]_\mathfrak{g})$ be any Lie algebra and $(V, \rho)$ be a representation of it. That is, $\rho : \mathfrak{g} \rightarrow \mathrm{End}(\mathfrak{g})$ is a Lie algebra homomorphism. Then a linear map $r: V \rightarrow \mathfrak{g}$ is said to be a {\em relative Rota-Baxter operator} \cite{jiang-sheng} if
\begin{align*}
[r(u), r(v)]_\mathfrak{g} = r (\rho_{r(u)} v - \rho_{r(v) } u), \text{ for all } u, v \in V.
\end{align*}

\quad (a) If $r: V \rightarrow \mathfrak{g}$ is a relative Rota-Baxter operator then its lift $\widetilde{r} : \mathfrak{g} \oplus V \rightarrow \mathfrak{g} \oplus V$ defined by $\widetilde{r} (x, u) := (r(u), 0)$ is a Nijenhuis operator on the semidirect product Lie algebra $\mathfrak{g} \ltimes V = (\mathfrak{g} \oplus V, [~,~]_\ltimes)$.

\quad (b) If $r_1, r_2: V \rightarrow \mathfrak{g}$ are two relative Rota-Baxter operators that are compatible (in the sense that the sum $r_1 + r_2$ is also a relative Rota-Baxter operator) and $r_2$ is invertible then $r_1 r_2^{-1}$ is a Nijenhuis operator.
\end{exam}

\begin{defn}
A {\bf Nijenhuis Lie algebra} is a Lie algebra $(\mathfrak{g}, [~, ~]_\mathfrak{g})$ endowed with a distinguished Nijenhuis operator $N : \mathfrak{g} \rightarrow \mathfrak{g}$ on it.

A Nijenhuis Lie algebra as above is denoted by the triple $(\mathfrak{g}, [~,~]_\mathfrak{g}, N)$ or simply by $(\mathfrak{g}, N)$ when the Lie bracket of $\mathfrak{g}$ is clear from the context.
\end{defn}

%Let  $(\mathfrak{g}, [~,~]_\mathfrak{g}, N)$ and  $(\mathfrak{g}', [~,~]_{\mathfrak{g}'}, N')$ be two Nijenhuis Lie algebras. 
A {\em homomorphism} of Nijenhuis Lie algebras from  $(\mathfrak{g}, [~,~]_\mathfrak{g}, N)$ to  $(\mathfrak{h}, [~,~]_{\mathfrak{h}}, S)$ is given by a Lie algebra homomorphism $\Phi : \mathfrak{g} \rightarrow \mathfrak{h}$ satisfying additionally $S \circ \Phi = \Phi \circ N$. Further, it is said to be an {\em isomorphism} if $\Phi$ is a linear isomorphism. We denote the group of all Nijenhuis Lie algebra automorphisms of  $(\mathfrak{g}, [~,~]_\mathfrak{g}, N)$ by the notation $\mathrm{Aut} (\mathfrak{g}, N)$.

\begin{defn}
Let $(\mathfrak{g}, [~,~]_\mathfrak{g}, N)$ be a Nijenhuis Lie algebra. A {\bf Nijenhuis representation} of $(\mathfrak{g}, [~,~]_\mathfrak{g}, N)$ is given by a triple $(V, \rho, S)$, where $(V, \rho)$ is a usual representation of the Lie algebra $(\mathfrak{g}, [~,~]_\mathfrak{g})$ and $S : V \rightarrow V$ is a linear map satisfying
\begin{align*}
    \rho_{N(x) } S (v) =  S \big(  \rho_{N(x)} v + \rho_x S (v) - S (\rho_x v)   \big), \text{ for } x \in \mathfrak{g}, v \in V.
\end{align*}
\end{defn}

\begin{exam}
(i) Let $(\mathfrak{g}, [~,~]_\mathfrak{g})$ be a Lie algebra and $(V, \rho)$ be a representation of it. Then the triple $(V, \rho, \mathrm{Id}_V)$ is a Nijenhuis representation of the Nijenhuis Lie algebra $(\mathfrak{g}, [~,~]_\mathfrak{g}, \mathrm{Id}_\mathfrak{g})$.

\medskip

(ii) Any Nijenhuis Lie algebra $(\mathfrak{g}, [~,~]_\mathfrak{g}, N)$ can be regarded as a Nijenhuis representation $(\mathfrak{g}, \rho_\mathrm{ad}, N)$ of itself, where $\rho_\mathrm{ad} : \mathfrak{g} \rightarrow \mathrm{End} (\mathfrak{g})$ is the adjoint representation given by $(\rho_\mathrm{ad})_x y := [x, y]_\mathfrak{g}$, for $x, y \in \mathfrak{g}$.

\medskip

(iii) Let $(\mathfrak{g}, [~,~]_\mathfrak{g}, N)$ be a Nijenhuis Lie algebra and $(V, \rho, S)$ be a Nijenhuis representation of it. Then for any $k \geq 0$, the triple $(V, \rho, S^k)$ is a Nijenhuis representation of the Nijenhuis Lie algebra $ (\mathfrak{g}, [~,~]_\mathfrak{g}, N^k)$.

\medskip

(iv) Let $(\mathfrak{g}, [~,~]_\mathfrak{g}, N)$ and $(\mathfrak{h}, [~,~]_\mathfrak{h}, S)$ be two Nijenhuis Lie algebras, and $\Phi : \mathfrak{g} \rightarrow \mathfrak{h}$ be a homomorphism of Nijenhuis Lie algebras. Then the triple $(\mathfrak{h}, \rho_\Phi, S)$ becomes a Nijenhuis representation of the Nijenhuis Lie algebra $(\mathfrak{g}, [~,~]_\mathfrak{g}, N)$, where $\rho_\Phi : \mathfrak{g} \rightarrow \mathrm{End} (\mathfrak{h})$ is given by $(\rho_\Phi)_x h := [\Phi (x), h]_\mathfrak{h}$, for $x \in \mathfrak{g}$ and $h \in \mathfrak{h}$. This Nijenhuis representation is said to be induced by the homomorphism $\Phi$.

\medskip

(v) Let $(V, \rho_V, S)$ and $(W, \rho_W, N_W)$ be two Nijenhuis representations of the Nijenhuis Lie algebra $(\mathfrak{g}, [~,~]_\mathfrak{g}, N)$. If $S$ and $N_W$ are both projections (i.e. $S^2 =S$ and $(N_W)^2 = N_W$) then the triple $(V \otimes W, \rho_{V \otimes W}, S \otimes N_W)$ is also a Nijenhuis representation of the Nijenhuis Lie algebra $(\mathfrak{g}, [~,~]_\mathfrak{g}, N)$, where 
\begin{align*}
(\rho_{V \otimes W})_x (v \otimes w) = (\rho_V)_x v \otimes w + v \otimes (\rho_W)_x w, \text{ for } x \in \mathfrak{g} \text{ and } v \otimes w \in V \otimes W.
\end{align*}

(vi) Let $(\mathfrak{g}, [~,~]_\mathfrak{g}, N)$ be a Nijenhuis Lie algebra and $(V, \rho)$ be a representation of the underlying Lie algebra $(\mathfrak{g}, [~,~]_\mathfrak{g})$. A linear map $ \zeta : V \rightarrow V$ is said to be {\em admissible} with respect to the map $\rho$ if
\begin{align*}
\zeta (\rho_{N (x)} v) + \rho_x \zeta^2 (v) = \rho_{N (x)} \zeta (v) + \zeta (\rho_x \zeta (v)), \text{ for all } x \in \mathfrak{g}, v \in V.
\end{align*}
In this case, the triple $(\mathfrak{g}^*, \rho_{\mathrm{coad}}, \zeta^*)$ is a Nijenhuis representation of the Nijenhuis Lie algebra $(\mathfrak{g}, [~,~]_\mathfrak{g}, N)$, where $\rho_\mathrm{coad} : \mathfrak{g} \rightarrow \mathrm{End} (\mathfrak{g}^*)$ is the coadjoint representation given by $(\rho_\mathrm{coad})_x (\alpha) (y) = - \alpha ([x, y]_\mathfrak{g})$, for $x, y \in \mathfrak{g}$ and $\alpha \in \mathfrak{g}^*$.

\medskip

(vii) A {\em relative Rota-Baxter Lie algebra} is a triple $((\mathfrak{g}, [~,~]_\mathfrak{g}), (V, \rho), r)$ consisting of a Lie algebra $(\mathfrak{g}, [~,~]_\mathfrak{g})$, a representation $(V, \rho)$ and a relative Rota-Baxter operator $r : V \rightarrow \mathfrak{g}$. It follows from Example \ref{nij-op-exam} (vi) that a relative Rota-Baxter Lie algebra $((\mathfrak{g}, [~,~]_\mathfrak{g}), (V, \rho), r)$ gives rise to a Nijenhuis Lie algebra $(\mathfrak{g} \oplus V, [~,~]_\ltimes, \widetilde{r})$. Recall that a {\em representation} \cite{jiang-sheng} of a relative Rota-Baxter Lie algebra $((\mathfrak{g}, [~,~]_\mathfrak{g}), (V, \rho), r)$ is a quadruple $((\mathfrak{h}, \rho_\mathfrak{h}), (W, \rho_W), \mu, s )$, where $(\mathfrak{h}, \rho_\mathfrak{h})$ and $(W, \rho_W)$ are both usual representations of the Lie algebra $(\mathfrak{g}, [~,~]_\mathfrak{g})$, and $\mu : \mathfrak{h} \rightarrow \mathrm{Hom} (V, W)$, $s : W \rightarrow \mathfrak{h}$ are linear maps that satisfy the following conditions:
\begin{align*}
\mu_{ (\rho_\mathfrak{h})_x h} v =~& (\rho_W)_x \mu_h v - \mu_h \rho_x v,\\
(\rho_\mathfrak{h})_{r(v)} s (w) =~& s \big(  (\rho_W)_{r (v)} w - \mu_{s (w)} v   \big),
\end{align*}
for all $x \in \mathfrak{g}$, $h \in \mathfrak{h}$, $v \in V$ and $w \in W$. In this case, we define linear maps $\sigma : \mathfrak{g} \oplus V \rightarrow \mathrm{End} (\mathfrak{h} \oplus W)$ and $\widetilde{s} : \mathfrak{h} \oplus W \rightarrow \mathfrak{h} \oplus W$ by
\begin{align*}
\sigma_{(x, v)} (h, w) := \big(  (\rho_\mathfrak{h})_x h ~ \! , ~ \! (\rho_W)_x w - \mu_h v  \big) ~~~~~ \text{ and } ~~~~ \widetilde{s} (h, w) = (s(w), 0),
\end{align*}
for $(x, v) \in \mathfrak{g} \oplus V$ and $(h, w) \in \mathfrak{h} \oplus W$. Then it can be checked that the triple $(\mathfrak{h} \oplus W, \sigma, \widetilde{s})$ is a Nijenhuis representation of the Nijenhuis Lie algebra $(\mathfrak{g} \oplus V, [~,~]_\ltimes, \widetilde{r})$.
\end{exam}

The next result shows that the standard semidirect product construction can be generalized to the context of Nijenhuis Lie algebras straightforwardly.

\begin{prop}
    Let $(\mathfrak{g}, [~,~]_\mathfrak{g}, N)$ be a Nijenhuis Lie algebra and $(V, \rho, S)$ be a Nijenhuis representation of it. Then $(\mathfrak{g} \oplus V, [~,~]_\ltimes, N \oplus S)$ is also a Nijenhuis Lie algebra, called the semidirect product.
\end{prop}

\medskip 

\section{Non-abelian cohomology of Nijenhuis Lie algebras}\label{sec3}
In this section, we first introduce the non-abelian cohomology of a Nijenhuis Lie algebra with values in another Nijenhuis Lie algebra. We show that the non-abelian cohomology group parametrizes the isomorphism classes of all non-abelian extensions of Nijenhuis Lie algebras. In particular, we show that the set of all isomorphism classes of abelian extensions of a Nijenhuis Lie algebra by a given Nijenhuis representation has a bijection with the second cohomology group.

\medskip

%Let $(\mathfrak{g}, [~,~]_\mathfrak{g}, N)$ and $(\mathfrak{h}, [~,~]_\mathfrak{h}, S)$ be two Nijenhuis Lie algebras.
\begin{defn} 
A {\bf non-abelian $2$-cocycle} of a Nijenhuis Lie algebra $(\mathfrak{g}, [~,~]_\mathfrak{g}, N)$ with values in another Nijenhuis Lie algebra $(\mathfrak{h}, [~,~]_\mathfrak{h}, S)$ is a triple $(\chi, \psi, F)$ consisting of linear maps $\chi : \wedge^2 \mathfrak{g} \rightarrow \mathfrak{h}$, $\psi : \mathfrak{g} \rightarrow \mathrm{Der} (\mathfrak{h})$ and $F: \mathfrak{g} \rightarrow \mathfrak{h}$ such that for all $x, y , z\in \mathfrak{g}$ and $h \in \mathfrak{h}$, the following set of identities are hold:
\begin{align}
&\psi_x \psi_y (h) - \psi_y \psi_x (h) - \psi_{[x, y]_\mathfrak{g}} (h) = [\chi (x, y), h]_\mathfrak{h}, \label{nc1}
\end{align}
\begin{align}
&\psi_x \chi (y, z) + \psi_y \chi (z, x) + \psi_z \chi (x, y) - \chi ([x, y]_\mathfrak{g} , z) - \chi ([y, z]_\mathfrak{g}, x) - \chi ([z, x]_\mathfrak{g}, y) = 0,  \label{nc2} 
\end{align}
\begin{align}
\psi_{N(x)} S (h) = S \big(  \psi_{N(x)} h + \psi_x S(h) - S (\psi_x h)   \big) + S [F(x), h]_\mathfrak{h} - [F(x), S(h)]_\mathfrak{h},  \label{nc3}
\end{align}
\begin{align}
& \chi (N (x), N (y)) - S \big( \chi ( N (x), y) + \chi (x, N (y)) - S ( \chi ( x, y) ) \big) - F \big( [N (x), y]_\mathfrak{g} + [x, N (y)]_\mathfrak{g} - N [x, y]_\mathfrak{g}    \big) \label{nc4} \\ \medskip \medskip \medskip
& \qquad \quad + \psi_{N(x)} F (y) - \psi_{N (y)} F (x) - S \big( \psi_x F (y) - \psi_y F (x) - F [x, y]_\mathfrak{g} \big) + [F (x), F (y)]_\mathfrak{h} = 0. \nonumber
\end{align}
\end{defn}

\medskip

Let $(\chi, \psi, F)$ and $(\chi', \psi', F')$ be two non-abelian $2$-cocycles of the Nijenhuis Lie algebra $(\mathfrak{g}, [~,~]_\mathfrak{g}, N)$ with values in the Nijenhuis Lie algebra $(\mathfrak{h}, [~,~]_\mathfrak{h}, S)$. They are said to be {\bf equivalent} if there exists a linear map $\varphi : \mathfrak{g} \rightarrow \mathfrak{h}$ such that
\begin{align}
\psi_x h - \psi'_x h =~& [\varphi (x), h]_\mathfrak{h},   \label{nce1}\\
\chi (x, y) - \chi' (x, y) =~&  \psi'_x (\varphi (y)) - \psi'_y (\varphi (x)) - \varphi ([x, y]_\mathfrak{g}) + [\varphi (x), \varphi (y)]_\mathfrak{h},  \label{nce2} \\
F(x) - F' (x) =~& S (\varphi (x)) - \varphi (N (x)), \label{nce3}
\end{align}
for all $x, y \in \mathfrak{g}$ and $h \in \mathfrak{h}$. We denote the set of all equivalence classes of such non-abelian $2$-cocycles by $H^2_\mathrm{nab} ((\mathfrak{g}, N); (\mathfrak{h}, S))$, and call it the {\bf non-abelian cohomology group} of the Nijenhuis Lie algebra $(\mathfrak{g}, [~,~]_\mathfrak{g}, N)$ with values in the Nijenhuis Lie algebra $(\mathfrak{h}, [~,~]_\mathfrak{h}, S)$.

\begin{remark}
It is important to remark that the non-abelian cohomology of a Lie algebra with values in another Lie algebra was considered in \cite{inas}, \cite{fre}. More precisely, a non-abelian $2$-cocycle of the Lie algebra $(\mathfrak{g}, [~,~]_\mathfrak{g})$ with values in the Lie algebra $(\mathfrak{h}, [~,~]_\mathfrak{h})$ is a pair $(\chi, \psi)$ of linear maps $\chi : \wedge^2 \mathfrak{g} \rightarrow \mathfrak{h}$ and $\psi : \mathfrak{g} \rightarrow \mathrm{Der} (\mathfrak{h})$ satisfying the identities (\ref{nc1}) and (\ref{nc2}). Hence by taking $N = \mathrm{Id}_\mathfrak{g}$ and $S = \mathrm{Id}_\mathfrak{h}$ in the identities (\ref{nc3}) and (\ref{nc4}), we get that a triple $(\chi, \psi, F)$ is a non-abelian $2$-cocycle of the Nijenhuis Lie algebra $(\mathfrak{g}, [~,~]_\mathfrak{g}, \mathrm{Id}_\mathfrak{g})$ with values in the Nijenhuis Lie algebra $(\mathfrak{h}, [~,~]_\mathfrak{h}, \mathrm{Id}_\mathfrak{h})$ if and only if $(\chi, \psi)$ is a non-abelian $2$-cocycle of the Lie algebra $(\mathfrak{g}, [~,~]_\mathfrak{g})$ with values in $(\mathfrak{h}, [~,~]_\mathfrak{h})$ and $F : \mathfrak{g} \rightarrow \mathfrak{h}$ is a linear map whose image is an abelian Lie subalgebra of $(\mathfrak{h}, [~,~]_\mathfrak{h}).$
\end{remark}

\medskip

We now turn our attention to non-abelian extensions of Nijenhuis Lie algebras. More precisely, we have the following.

\begin{defn}
A {\bf non-abelian extension} of a Nijenhuis Lie algebra $(\mathfrak{g}, [~,~]_\mathfrak{g}, N)$ by another Nijenhuis Lie algebra $(\mathfrak{h}, [~,~]_\mathfrak{h}, S)$ is a short exact sequence of Nijenhuis Lie algebras of the form
\begin{align}\label{extension}
\xymatrix{
0 \ar[r] & (\mathfrak{h}, [~,~]_\mathfrak{h}, S) \ar[r]^i & (\mathfrak{e}, [~,~]_\mathfrak{e}, U) \ar[r]^p & (\mathfrak{g}, [~,~]_\mathfrak{g}, N) \ar[r] & 0.
}
\end{align}
\end{defn}
\noindent We often denote the non-abelian extension as above simply by the Nijenhuis Lie algebra $(\mathfrak{e}, [~,~]_\mathfrak{e}, U)$ when the short exact sequence is clear from the context.

Let $(\mathfrak{e}, [~,~]_\mathfrak{e}, U)$ and $(\mathfrak{e}', [~,~]_{\mathfrak{e}'}, U')$ be two non-abelian extensions. They are said to be {\bf isomorphic} if there exists an isomorphism $\Phi : \mathfrak{e} \rightarrow \mathfrak{e}'$ of Nijenhuis Lie algebras making the following diagram commutative:
\begin{align}\label{iso-extension}
\xymatrix{
0 \ar[r] & (\mathfrak{h}, [~,~]_\mathfrak{h}, S) \ar@{=}[d] \ar[r]^i & (\mathfrak{e}, [~,~]_\mathfrak{e}, U) \ar[d]^\Phi \ar[r]^p & (\mathfrak{g}, [~,~]_\mathfrak{g}, N) \ar@{=}[d] \ar[r] & 0 \\ 
0 \ar[r] & (\mathfrak{h}, [~,~]_\mathfrak{h}, S) \ar[r]_{i'} & (\mathfrak{e}', [~,~]_{\mathfrak{e}'}, U') \ar[r]_{p'} & (\mathfrak{g}, [~,~]_\mathfrak{g}, N) \ar[r] & 0.
}
\end{align}
We denote the set of all isomorphism classes of non-abelian extensions of the Nijenhuis Lie algebra $(\mathfrak{g}, [~,~]_\mathfrak{g}, N)$ by the Nijenhuis Lie algebra $(\mathfrak{h}, [~,~]_\mathfrak{h}, S)$ simply by the notation $\mathcal{E}xt_{nab} ((\mathfrak{g}, N); (\mathfrak{h}, S))$. Our first result of this paper gives a parametrization of $\mathcal{E}xt_{nab} ((\mathfrak{g}, N); (\mathfrak{h}, S))$ by the non-abelian cohomology group defined above. More precisely, we have the following.

\begin{thm}\label{main-thm-sec3}
Let $(\mathfrak{g}, [~,~]_\mathfrak{g}, N)$ and $(\mathfrak{h}, [~,~]_\mathfrak{h}, S)$ be two Nijenhuis Lie algebras. Then there is a bijection
\begin{align*}
\mathcal{E}xt_{nab} ((\mathfrak{g}, N); (\mathfrak{h}, S)) ~ \! \cong ~ \! H^2_{nab}  ((\mathfrak{g}, N); (\mathfrak{h}, S)). 
\end{align*}
\end{thm}

\begin{proof}
Let $(\mathfrak{e}, [~,~]_\mathfrak{e}, U)$ be a non-abelian extension as of (\ref{extension}). Choose any section $s:\mathfrak{g} \rightarrow \mathfrak{e}$ of the map $p$ (i.e. $s$ is a linear map satisfying $ps = \mathrm{Id}_\mathfrak{g}$). Then we define linear maps $\chi : \wedge^2 \mathfrak{g} \rightarrow \mathfrak{h}$, $\psi : \mathfrak{g} \rightarrow \mathrm{Der}(\mathfrak{h})$ and $F : \mathfrak{g} \rightarrow \mathfrak{h}$ respectively by
\begin{align*}
\chi (x, y) := [s(x), s(y)]_\mathfrak{e} - s [x, y]_\mathfrak{g}, \qquad \psi_x h := [s(x) , h]_\mathfrak{e} \quad \text{ and } \quad F(x) := (Us - s N)(x),
\end{align*}
for all $x, y \in \mathfrak{g}$ and $h \in \mathfrak{h}$. It has been observed in \cite{fre} that the maps $\chi$ and $\psi$ satisfy the identities (\ref{nc1}) and (\ref{nc2}). For any $x \in \mathfrak{g}$ and $h \in \mathfrak{h}$, we observe that
\begin{align*}
&\psi_{N (x)} S (h) - S ( \psi_{N (x)} h + \psi_x S(h) - S (\psi_x h)) - S [F(x), h]_\mathfrak{h} + [F (x), S (h)]_\mathfrak{h} \\
&= \cancel{ [sN (x), S (h)]_\mathfrak{e}} - S \big(  \cancel{[sN (x), h]_\mathfrak{e}} + [s(x), S (h)]_\mathfrak{e}  - S [s(x), h]_\mathfrak{e} \big)\\
& \qquad \quad - S (  [Us(x), h]_\mathfrak{e} - \cancel{[sN (x), h]_\mathfrak{e} } ) + [Us(x), S(h)]_\mathfrak{e} - \cancel{ [s N (x), S(h)]_\mathfrak{e} } \\
&= - U \big( [s(x), U (h)]_\mathfrak{e} + [Us (x), h]_\mathfrak{e} - U [s(x), h]_\mathfrak{e} \big) + [Us (x), U(h)]_\mathfrak{e} \qquad (\text{as } S = U |_\mathfrak{h})\\
&= 0 \quad (\text{as } U \text{ is a Nijenhuis operator}).
\end{align*}
Hence the identity (\ref{nc3}) also holds. Moreover, for any $x, y \in \mathfrak{g}$,
\begin{align*}
&\chi (N(x), N (y)) - S \big(  \chi (N (x), y) + \chi (x, N (y)) - S \chi (x, y)  \big) - F \big( [N(x), y]_\mathfrak{g} + [x, N (y) ]_\mathfrak{g} - N [x, y]_\mathfrak{g}  \big) \\
& \qquad \qquad + \psi_{N (x)} F (y) - \psi_{N(y)} F (x) - S \big(  \psi_x F (y) - \psi_y F (x) - F [x, y]_\mathfrak{g} \big) + [F (x), F (y)]_\mathfrak{h} \\
&= [s N (x) , s N (y)]_\mathfrak{e} \underbrace{-s [N (x), N(y)]_\mathfrak{g}}_{(A)} - S ( [s N (x), s(y)]_\mathfrak{e} - s [N (x), y]_\mathfrak{g}) - S ( [s(x), s N (y)]_\mathfrak{e} - s [x, N (y)]_\mathfrak{g}) \\
& \quad + S^2 ( \underbrace{[s(x), s(y)]_\mathfrak{e}}_{(B)} - s [x, y]_\mathfrak{g}) - Us [N (x), y]_\mathfrak{g} + \underbrace{sN [N (x), y]_\mathfrak{g}}_{(A)} - Us [x, N (y)]_\mathfrak{g} + \underbrace{ sN [x, N (y)]_\mathfrak{g}}_{(A)} \\
& \quad + UsN [x, y]_\mathfrak{g} \underbrace{ - sN^2 [x, y]_\mathfrak{g}}_{(A)} + [sN (x), Us (y)]_\mathfrak{e} - [s N (x), s N (y)]_\mathfrak{e} - [sN (y), Us (x)]_\mathfrak{e} + [sN (y), sN (x)]_\mathfrak{e} \\
& \quad - S  \big(  \underbrace{[s(x), Us (y)]_\mathfrak{e}}_{(B)} - [s(x), s N (y)]_\mathfrak{e} \underbrace{- [s(y), Us (x)]_\mathfrak{e}}_{(B)} + [s(y), s N (x)]_\mathfrak{e} - Us [x, y]_\mathfrak{g} + s N [x, y]_\mathfrak{g} \big) \\
& \quad  + \underbrace{ [Us (x), Us (y)]_\mathfrak{e}}_{(B)} - [Us (x), s N (y)]_\mathfrak{e} - [sN (x), Us (y)]_\mathfrak{e} + [sN (x), s N (y)]_\mathfrak{e}.
\end{align*}
In the above expression, the terms underlined with (A) are cancelled an $N$ is a Nijenhuis operator on $(\mathfrak{g}, [~,~]_\mathfrak{g})$, and the terms underlined with (B) are also cancelled as $U$ is a Nijenhuis operator on $(\mathfrak{e}, [~,~]_\mathfrak{e})$. All the other terms are cancelled with each other. (Note that, while cancelling the terms, we often need to use that $S= U |_\mathfrak{h}$). Hence the entire above expression vanishes and thus verifies the identity (\ref{nc4}). This shows that $(\chi, \psi, F)$ is a non-abelian $2$-cocycle. However, this non-abelian $2$-cocycle depends on the section $s$. For any other section $\widetilde{s} : \mathfrak{g} \rightarrow \mathfrak{e}$, let $(\widetilde{\chi}, \widetilde{\psi}, \widetilde{F})$ be the non-abelian $2$-cocycle produced from the same non-abelian extension $(\mathfrak{e}, [~,~]_\mathfrak{e}, U)$. We define a map $\varphi : \mathfrak{g} \rightarrow \mathfrak{h}$ by $\varphi (x) := s(x) - \widetilde{s} (x)$, for $x \in \mathfrak{g}$. Then we observe that
\begin{align*}
\psi_x (h) - \widetilde{\psi}_x (h) =~& [s(x) , h ]_\mathfrak{e} - [\widetilde{s}(x), h]_\mathfrak{e} = [\varphi (x), h]_\mathfrak{h},\\ \medskip \medskip
\chi (x, y) - \widetilde{\chi} (x, y) =~& [s(x), s(y)]_\mathfrak{e} - s [x, y]_\mathfrak{g} - [\widetilde{s} (x) , \widetilde{s} (y)]_\mathfrak{e} + \widetilde{s} [x, y]_\mathfrak{g} \\
=~& [ (\varphi + \widetilde{s})(x), (\varphi + \widetilde{s}) (y)]_\mathfrak{e} - (\varphi + \widetilde{s}) [x, y]_\mathfrak{g} - [ \widetilde{s}(x), \widetilde{s} (y)]_\mathfrak{e} + \widetilde{s} [x, y]_\mathfrak{g} \\
=~& \widetilde{\psi}_x ( \varphi (y)) - \widetilde{\psi}_y (\varphi (x)) - \varphi [x, y]_\mathfrak{g} + [ \varphi (x), \varphi (y)]_\mathfrak{h}, \\ \medskip \medskip
F(x) - \widetilde{F} (x) =~& (Us -s N) (x) - (U \widetilde{s} - \widetilde{s} N) (x) \\
=~& U (s - \widetilde{s}) (x) - (s - \widetilde{s}) N (x) = S (\varphi (x))- \varphi (N (x)),
\end{align*}
for all $x, y \in \mathfrak{g}$ and $h \in \mathfrak{h}$. This shows that the non-abelian $2$-cocycles $(\chi, \psi, F)$ and $(\widetilde{\chi}, \widetilde{\psi}, \widetilde{F})$ are equivalent. Hence they corresponds to the same element in $H^2_{nab} (( \mathfrak{g}, N); (\mathfrak{h}, S)).$ Next, let $(\mathfrak{e}, [~,~]_\mathfrak{e}, U)$ and $(\mathfrak{e}', [~,~]_{\mathfrak{e}'}, U')$ be two isomorphic non-abelian extensions as of (\ref{iso-extension}). For any section $s$ of the map $p$, the composition $s' = \Phi \circ s$ is a section of the map $p'$. If $(\chi', \psi', F')$ is the non-abelian $2$-cocycle produced from the non-abelian extension $(\mathfrak{e}', [~,~]_{\mathfrak{e}'}, U')$ then
\begin{align*}
\chi' (x, y) =~& [\Phi \circ s (x) , \Phi \circ s (y)]_{\mathfrak{e}'} - \Phi \circ s [x, y]_\mathfrak{g} \\
=~& \Phi  ( [s(x), s(y)]_\mathfrak{e} - s [x, y]_\mathfrak{g} )  = \chi (x, y) \quad (\text{as } \Phi|_\mathfrak{h} = \mathrm{Id}_\mathfrak{h}),\\
\psi'_x (h) =~& [\Phi \circ s (x), h]_{\mathfrak{e}'} \\
=~& [ \Phi (s (x)), \Phi (h)]_{\mathfrak{e}'} = \Phi [s(x), h]_\mathfrak{e} = \psi_x (h)  \quad (\text{as } \Phi|_\mathfrak{h} = \mathrm{Id}_\mathfrak{h}), \\
F' (x) =~& ( U (\Phi \circ s) - (\Phi \circ s) N)(x) \\
=~&  \Phi ( (Us - sN)(x)) = F(x) \quad (\text{as } \Phi |_\mathfrak{h} = \mathrm{Id}_\mathfrak{h}),
\end{align*}
for all $x, y \in \mathfrak{g}$ and $h \in \mathfrak{h}$. This shows that $(\chi, \psi, F) = (\chi', \psi', F')$. As a conclusion, we obtain a well-defined map $\Upsilon : \mathcal{E}xt_{nab}((\mathfrak{g}, N); (\mathfrak{h}, S) ) \rightarrow H^2_{nab} ((\mathfrak{g}, N); (\mathfrak{h}, S) ).$

\medskip

To obtain a map in the other direction, we begin with a non-abelian $2$-cocycle $(\chi, \psi, F)$. Consider the vector space $\mathfrak{g} \oplus \mathfrak{h}$ equipped with the bilinear skew-symmetric bracket
\begin{align*}
[(x, h), (y, k)]_{\chi, \psi} := \big(     [x, y]_\mathfrak{g} ~ \! , ~ \! \psi_x k - \psi_y h + \chi (x, y) + [h, k]_\mathfrak{h} \big),
\end{align*}
for $(x, h), (y, k) \in \mathfrak{g} \oplus \mathfrak{h}$. Then it follows from (\ref{nc1}) and (\ref{nc2}) that $(\mathfrak{g} \oplus \mathfrak{h}, [~, ~]_{\chi, \psi})$ is a Lie algebra (see, for example \cite{fre}). We also define a linear map $U_F : \mathfrak{g} \oplus \mathfrak{h} \rightarrow \mathfrak{g} \oplus \mathfrak{h}$ by
\begin{align*}
U_F (x, h) := \big( N (x) ~ \! , ~ \! S (h) + F (x) \big), \text{ for } (x, h) \in \mathfrak{g} \oplus \mathfrak{h}. 
\end{align*}
For any $(x, h) , (y, k) \in \mathfrak{g} \oplus \mathfrak{h}$, we observe that
\begin{align*}
&[ U_F (x, h), U_F (y, k) ]_{\chi, \psi} \\
&= \big[ \big( N(x), S(h) + F (x) \big), \big( N(y), S(k) + F (y) \big) \big]_{\chi, \psi} \\
&= \big(  [ N(x), N(y)]_\mathfrak{g} ~ \!, ~ \underbrace{\psi_{N(x)} S(k)}_{(A1)} + \underbrace{\psi_{N(x)} F (y)}_{(A2)}  \underbrace{- \psi_{N(y)} S(h)}_{(A3)}  \underbrace{- \psi_{N(y)} F (x)}_{(A4)} + \underbrace{\chi (N(x), N(y))}_{(A5)} \\
& \qquad \qquad \qquad \qquad  \quad + \underbrace{[S(h), S(k)]_\mathfrak{h} }_{(A6)} + \underbrace{[S(h), F (y)]_\mathfrak{h}}_{(A7)} + \underbrace{[F(x), S(k)]_\mathfrak{h}}_{(A8)} + \underbrace{[ F (x), F (y)]_\mathfrak{h}}_{(A9)} \big) \\
&= \bigg(   N[N(x),y]_\mathfrak{g} + N[x, N(y)]_\mathfrak{g} -N^2 [x, y]_\mathfrak{g} ~ \! , ~ \! \underbrace{S ( \psi_{N(x)} k + \psi_{x} S(k) - S (\psi_x k) ) + S [F (x), k]_\mathfrak{h}}_{=~ (A1) + (A8)} \\
& \quad \quad \underbrace{- S \big( \psi_{N(y)} h + \psi_{y} S(h) - S(\psi_y h) \big) - S [F(y), h]_\mathfrak{h}}_{=~(A3) + (A7)} +  \underbrace{S [S(h), k]_\mathfrak{h} + S [h, S(k)]_\mathfrak{h} - S^2 [h, k]_\mathfrak{h}}_{= ~(A6)} \\
& \quad \quad \underbrace{ + S \big( \chi (N(x), y) +  \chi (x, N(y) ) - S \chi (x, y) \big) + F ([N(x), y]_\mathfrak{g} + [x, N(y)]_\mathfrak{g} - N [x, y]_\mathfrak{g} ) }_{=~ (A2)+(A4)+(A5)+(A9) ~ \text{ continued to the next line}}\\
& \qquad \quad  \underbrace{ + S \big( \psi_{x} F (y) - \psi_{y} F (x) - F [x, y]_\mathfrak{g} \big)}_{ =~ (A2)+(A4)+(A5)+(A9)    } \bigg) \\
&= \bigg( N[N(x),y]_\mathfrak{g} ~ \! ,~ \! S \big( \psi_{N(x)} k - \psi_{y} S(h) - \psi_{y} F (x) + \chi (N(x), y) + [S(h) + F (x), k]_\mathfrak{h}  \big) + F [N(x), y]_\mathfrak{g} \bigg) \\
& \quad + \bigg( N[x, N(y)]_\mathfrak{g} ~ \!, ~ \! S \big( \psi_{x} S(k) + \psi_{x} F (y) - \psi_{N(y)} h + \chi (x, N(y)) + [h, S (k)+ F(y)]_\mathfrak{h} \big) + F [x, N(y)]_\mathfrak{g} \bigg) \\ 
& \quad - \bigg(  N^2 [x, y]_\mathfrak{g} ~ \! , ~ \! S^2 (\psi_x k - \psi_y h + \chi (x, y) + [h, k]_\mathfrak{h}) + SF [x, y]_\mathfrak{g} + FN [x, y]_\mathfrak{g}   \bigg) \\  
&=  U_F \big( [N(x),y]_\mathfrak{g} ~ \!, ~ \! \psi_{N(x)} k - \psi_{y} S(h) - \psi_{y} F (x) + \chi (N(x), y) + [S(h) + F (x), k]_\mathfrak{h})   \big) \\
& \quad + U_F \big( [x, N(y)]_\mathfrak{g} ~ \!, ~ \!  \psi_{x} S(k) + \psi_{x} F (y) - \psi_{N(y)} h + \chi (x, N(y)) + [h, S (k)+ F(y)]_\mathfrak{h}  \big) \\
& \quad - U_F \big( N [x, y]_\mathfrak{g} ~ \! , ~ \!  S \big(  \psi_x k - \psi_y h + \chi (x, y) + [h, k]_\mathfrak{h} \big) + F [x, y]_\mathfrak{g}   \big) \\
&= U_F \big( [ (N(x), S(h) + F(x)), (y, k)]_{\chi, \psi} \big) ~+~ U_F \big( [(x, h), (N(y), S(k) + F (y))]_{\chi, \psi}   \big) - (U_F)^2 [(x, h), (y, k)]_{\chi, \psi} \\
%&= U_F \big(  [U_F ((x, h)), (y, k)]_{\chi, \psi} + [(x, h), U_F ((y,k))]_{\chi, \psi}  \big)\\
&= U_F \big(   [U_F (x, h), (y, k)]_{\chi, \psi} + [(x, h), U_F (y, k)]_{\chi, \psi} - U_F [ (x, h), (y, k)]_{\chi, \psi}   \big).
\end{align*}
This shows that $U_F$ is a Nijenhus operator on the above-constructed Lie algebra which in turn implies that $(\mathfrak{g} \oplus \mathfrak{h}, [~,~]_{\chi, \psi}, U_F)$ is a Nijenhuis Lie algebra. Moreover, this is a non-abelian extension of the Nijenhuis Lie algebra $(\mathfrak{g}, [~,~]_\mathfrak{g}, N)$ by the Nijenhuis Lie algebra $(\mathfrak{h}, [~,~]_\mathfrak{h}, S)$, where the structure maps $i$ and $p$ are respectively the inclusion and projection maps. Next, let $(\chi, \psi, F)$ and $(\chi', \psi', F')$ be two equivalent non-abelian $2$-cocycles. Therefore, there exists a linear map $\varphi : \mathfrak{g} \rightarrow \mathfrak{h}$ such that the identities (\ref{nce1}), (\ref{nce2}) and (\ref{nce3}) are hold. Let $(\mathfrak{g} \oplus \mathfrak{h}, [~,~]_{\chi, \psi}, U_F)$ and $(\mathfrak{g} \oplus \mathfrak{h}, [~,~]_{\chi', \psi'}, U_{F'})$ be the non-abelian extensions induced by the above non-abelian $2$-cocycles, respectively. We now define a linear isomorphism $\Phi : \mathfrak{g} \oplus \mathfrak{h} \rightarrow \mathfrak{g} \oplus \mathfrak{h}$ by
\begin{align*}
\Phi (x, h) = (x , \varphi (x) + h), \text{ for } (x, h) \in \mathfrak{g} \oplus \mathfrak{h}.
\end{align*}
Then we have
\begin{align*}
&\Phi \big(  [(x, h), (y, k)]_{\chi, \psi}  \big) \\
&= \Phi \big(  [x, y]_\mathfrak{g} ~ \!, ~ \! \psi_{x} k - \psi_{y} h + \chi (x,y) + [h,k]_\mathfrak{h}  \big) \\
&= \big( [x,y]_\mathfrak{g} ~ \!, ~ \! \varphi [x,y]_\mathfrak{g} + \psi_{x} k - \psi_{y} h + \chi (x,y) + [h,k]_\mathfrak{h}  \big) \\
&= \big(   [x,y]_\mathfrak{g} ~ \! , ~ \! \cancel{\varphi [x,y]_\mathfrak{g}} + \psi'_{x} k + [\varphi(x), k]_\mathfrak{h} - \psi'_{y} h - [\varphi (y), h]_\mathfrak{h} \\
& \quad \qquad \qquad + \chi' (x,y) + \psi'_{x} \varphi (y) - \psi'_{y} \varphi (x) - \cancel{\varphi [x,y]_\mathfrak{g}} + [\varphi (x), \varphi(y)]_\mathfrak{h} + [h,k]_\mathfrak{h} \big) \\
&= \big( [x,y]_\mathfrak{g} ~ \!, ~ \! \psi'_{x} \varphi (y)    + \psi'_{x} k - \psi'_{y} \varphi (x) - \psi'_{y} h + \chi' (x, y) + [\varphi (x) + h, \varphi (y) + k ]_\mathfrak{h} \big) \quad (\text{after rearranging}) \\
&= [ ( x, \varphi(x) + h),( y, \varphi(y) + k )]_{\chi', \psi'} \\
&= [\Phi (x,h), \Phi (y,k)]_{\chi', \psi'}
\end{align*}
and
\begin{align*}
(U_{F'} \circ \Phi) (x, h) &= U_{F'} (x, \varphi (x) + h ) \\
&= (N(x) ~ \!, ~ \! S (\varphi (x)) + S(h) + F' (x) )\\
&= (N(x) ~ \! , ~ \! \varphi (N(x)) + S (h) + F (x) ) = (\Phi \circ U_F) (x, h),
\end{align*}
for all $(x, h), (y, k) \in \mathfrak{g} \oplus \mathfrak{h}$. This shows that the map $\Phi$ makes the above non-abelian extensions isomorphic. As a consequence, we obtain a well-defined map $\Omega : H^2_{nab} ((\mathfrak{g}, N); (\mathfrak{h}, S)) \rightarrow \mathcal{E}xt_{nab} ((\mathfrak{g}, N); (\mathfrak{h}, S)).$

\medskip

Finally, it is not hard to see that the maps $\Upsilon$ and $\Omega$ are inverses to each other. Hence we obtain the required bijection. 
\end{proof}

\begin{remark}\label{remark-nl-def}
It is well-known that a Nijenhuis Lie algebra $(\mathfrak{g}, [~,~]_\mathfrak{g}, N)$ gives rise to a new Lie algebra structure on the vector space $\mathfrak{g}$ with the bracket
\begin{align*}
[x, y]_N := [N(x), y]_\mathfrak{g} + [x, N (y)]_\mathfrak{g} - N [x, y]_\mathfrak{g}, \text{ for } x, y \in \mathfrak{g}.
\end{align*}
The Lie algebra $(\mathfrak{g}, [~,~]_N)$ is often referred as the {\em deformed Lie algebra} and denoted simply by $\mathfrak{g}_N$. Let $(\mathfrak{g}, [~,~]_\mathfrak{g}, N)$ and $(\mathfrak{h}, [~,~]_\mathfrak{h}, S)$ be two Nijenhuis Lie algebras with the deformed Lie algebras $\mathfrak{g}_N = (\mathfrak{g}, [~,~]_N)$ and $ \mathfrak{h}_S = (\mathfrak{h}, [~,~]_S)$, respectively. Suppose $(\chi, \psi, F)$ is a non-abelian $2$-cocycle of the Nijenhuis Lie algebra $(\mathfrak{g}, [~,~]_\mathfrak{g}, N)$ with values in the Nijenhuis Lie algebra $(\mathfrak{h}, [~,~]_\mathfrak{h}, S)$. Consider the Nijenhuis Lie algebra $(\mathfrak{g} \oplus \mathfrak{h}, [~,~]_{\chi, \psi}, U_F)$ given in Theorem \ref{main-thm-sec3}. If $(\mathfrak{g} \oplus \mathfrak{h})_{U_F} = (\mathfrak{g} \oplus \mathfrak{h}, [~,~]_{U_F})$ is the corresponding deformed Lie algebra then we have
\begin{align}\label{deformed-cocycle}
[(x,  h), (y, k)]_{U_F} :=~& [ U_F (x, h), (y, k)]_{\chi, \psi} + [(x, h), U_F (y, k)]_{\chi, \psi} - U_F [(x, h), (y, k)]_{\chi, \psi} \nonumber \\
=~& \Big(  [N (x), y]_\mathfrak{g} + [x, N (y)]_{\mathfrak{g} } - N [x, y]_\mathfrak{g} ~ \!, ~ \! \psi_{N (x)} k - \psi_y S (h) - \psi_y F (x)  + \chi (N (x), y) \nonumber \\
& ~~~~ + [S(h) , k]_\mathfrak{h} + [F(x), k]_\mathfrak{h} + \psi_x S(k) + \psi_x F (y) - \psi_{N (y)} h + \chi (x, N (y)) + [h, S(k)]_\mathfrak{h} \nonumber \\ & ~~~~ + [h, F (y)]_\mathfrak{h} - S (\psi_x k) + S (\psi_y h) - S \chi (x, y) - S [h, k]_\mathfrak{h} - F [x, y]_\mathfrak{g} \Big) \nonumber \\
=~& \big(  [x, y]_N ~ \! , ~ \! \overline{\psi}_x k - \overline{\psi}_y h + \overline{\chi} (x, y) + [h, k]_S  \big), 
\end{align}
for $(x, h), (y, k) \in \mathfrak{g} \oplus \mathfrak{h}$. Here the maps $\overline{\chi} : \wedge^2 \mathfrak{g} \rightarrow \mathfrak{h}$ and $\overline{\psi} : \mathfrak{g} \rightarrow \mathrm{Der} (\mathfrak{h}_S)$ are respectively given by
\begin{align*}
\overline{\chi} (x, y) :=~& \chi (N(x), y) + \chi (x, N (y)) - S \chi (x, y) + \psi_x F (y) - \psi_y F (x) - F [x, y]_\mathfrak{g}, \\
\overline{\psi}_x (h) :=~&  \psi_{N (x)} h + \psi_{x} S (h) - S (\psi_x h) + [F(x), h]_\mathfrak{h},
\end{align*} for $x, y \in \mathfrak{g}$ and $h \in \mathfrak{h}$. It follows from the expression (\ref{deformed-cocycle}) that the pair $( \overline{\chi}, \overline{\psi})$ is a non-abelian $2$-cocycle of the deformed Lie algebra $\mathfrak{g}_N$ with values in the deformed Lie algebra $\mathfrak{h}_S$. 
%This shows that any non-abelian $2$-cocycle of Nijenhuis Lie algebra can be deformed to get a non-abelian $2$-cocycle of the deformed Lie algebra. 
Next, let $(\chi, \psi, F)$ and $(\chi', \psi', F')$ be two equivalent non-abelian $2$-cocycles of the Nijenhuis Lie algebra $(\mathfrak{g}, [~,~]_\mathfrak{g}, N)$ with values in the Nijenhuis Lie algebra $(\mathfrak{h}, [~,~]_\mathfrak{h}, S)$. That is, there is a linear map $\varphi : \mathfrak{g} \rightarrow \mathfrak{h}$ satisfying the identities (\ref{nce1})-(\ref{nce3}). If $(\overline{\chi}, \overline{\psi})$ and $(\overline{\chi'}, \overline{\psi'})$ are the corresponding non-abelian $2$-cocycles of the deformed Lie algebra $\mathfrak{g}_N$ with values in the deformed algebra $\mathfrak{h}_S$ then we can easily verify that
\begin{align*}
\overline{\psi}_x (h) - \overline{\psi'}_x (h) =~& [\varphi (x), h]_S, \\
\overline{\chi} (x, y) - \overline{\chi'} (x, y) =~& \overline{\psi'}_x (\varphi (y)) - \overline{\psi'}_y (\varphi (x)) - \varphi ([x , y]_N) + [\varphi (x), \varphi (y)]_S. 
\end{align*}
This shows that the non-abelian $2$-cocycles $(\overline{\chi}, \overline{\psi})$ and $(\overline{\chi'}, \overline{\psi'})$ are also equivalent by the map $\varphi$. Hence there is a well-defined map 
\begin{align*}
H^2_{nab} ((\mathfrak{g}, N); (\mathfrak{h}, S)) \rightarrow H^2_{nab} (\mathfrak{g}_N; \mathfrak{h}_S) , ~~ [(\chi, \psi, F)] \mapsto [(\overline{\chi}, \overline{\psi})]
\end{align*}
from the non-abelian cohomology of Nijenhuis Lie algebra to the non-abelian cohomology of the deformed Lie algebra.

\end{remark}

\medskip

\noindent {\bf Particular case: Abelian extensions.}  Here we shall focus on the abelian extensions of a Nijenhuis Lie algebra by a given Nijenhuis representation as a particular case of non-abelian extensions.

Let $(\mathfrak{g}, [~,~]_\mathfrak{g}, N)$ be a Nijenhuis Lie algebra and $(V, \rho, S)$ be a Nijenhuis representation of it. Then a {\bf $2$-cocycle} of the Nijenhuis Lie algebra $(\mathfrak{g}, [~,~]_\mathfrak{g}, N)$ with coefficients in the Nijenhuis representation $(V, \rho, S)$ is given by a pair $(\chi, F)$ of linear maps $\chi : \wedge^2 \mathfrak{g} \rightarrow V$ and $F: \mathfrak{g} \rightarrow V$ satisfying
\begin{align}\label{2co-abel1}
\rho_x \chi (y, z) + \rho_y \chi (z, x) + \rho_z \chi (x, y) - \chi ([x, y]_\mathfrak{g}, z) - \chi ( [y, z]_\mathfrak{g}, x) - \chi ([z, x]_\mathfrak{g}, y) = 0,
\end{align}
\begin{align}\label{2co-abel2}
 & \chi (N (x), N (y)) - S \big( \chi ( N (x), y) + \chi (x, N (y)) - S ( \chi ( x, y) ) \big) - F \big( [N (x), y]_\mathfrak{g} + [x, N (y)]_\mathfrak{g} - N [x, y]_\mathfrak{g}    \big) \\ \medskip \medskip \medskip
& \qquad \qquad \qquad \qquad + \rho_{N(x)} F (y) - \rho_{N (y)} F (x) - S \big( \rho_x F (y) - \rho_y F (x) - F [x, y]_\mathfrak{g} \big) = 0, \nonumber
\end{align}
for all $x, y , z \in \mathfrak{g}$. Let $(\chi, F)$ and $(\chi', F')$ be two $2$-cocycles. They are said to be {\bf cohomologous} if there exists a linear map $\varphi : \mathfrak{g} \rightarrow V$ such that
\begin{align*}
\chi (x, y) - \chi' (x, y) = \rho_x \varphi (y) - \rho_y \varphi (x) - \varphi ([x, y]_\mathfrak{g}) \quad \text{ and } \quad F (x) - F' (x) = S (\varphi (x) )- \varphi (N (x)),
\end{align*}
for all $x, y \in \mathfrak{g}$. We denote the set of all $2$-cocycles modulo the cohomologous relation by the notation $H^2 ((\mathfrak{g}, N); (V, S))$ and call it the {\bf second cohomology group} of the Nijenhuis Lie algebra $(\mathfrak{g}, [~,~]_\mathfrak{g}, N)$ with coefficients in the Nijenhuis representation $(V, \rho, S)$.

\begin{remark}
    Note that a $2$-cocycle of a Nijenhuis Lie algebra with coefficients in a Nijenhuis representation can be regarded as a non-abelian $2$-cocycle by looking at the pair $(V, S)$ as a Nijenhuis Lie algebra with the trivial Lie bracket on $V$. More precisely, a pair $(\chi, F)$ is a $2$-cocycle of the Nijenhuis Lie algebra $(\mathfrak{g}, [~,~]_\mathfrak{g}, N)$ with coefficients in the Nijenhuis representation $(V, \rho, S)$ if and only if the triple $(\chi, \rho, F)$ is a non-abelian $2$-cocycle of the Nijenhuis Lie algebra $(\mathfrak{g}, [~,~]_\mathfrak{g}, N)$ with values in the Nijenhuis Lie algebra $(V, [~,~]_V = 0, S)$.
    \end{remark}

Next, let $(\mathfrak{g}, [~,~]_\mathfrak{g}, N)$ be a Nijenhuis Lie algebra and $(V, S)$ be a pair consisting of a vector space endowed with a linear map. As before, we can regard $(V, S)$ as a Nijenhuis Lie algebra with the trivial Lie bracket on $V$. Then an {\bf abelian extension} of the Nijenhuis Lie algebra $(\mathfrak{g}, [~,~]_\mathfrak{g}, N)$ by the pair $(V, S)$ is a short exact sequence of Nijenhuis Lie algebras:
\begin{align}\label{abel-ext}
\xymatrix{
0 \ar[r] & (V, [~,~]_V = 0, S) \ar[r]^i & (\mathfrak{e}, [~,~]_\mathfrak{e}, U) \ar[r]^p & (\mathfrak{g}, [~,~]_\mathfrak{g}, N) \ar[r] & 0.
}
\end{align}
The notion of isomorphisms between such abelian extensions can be defined similarly. Given an abelian extension (\ref{abel-ext}), choose any section $s$ of the map $p$. As before, we define maps $\chi : \wedge^2 \mathfrak{g} \rightarrow V$, $\rho : \mathfrak{g} \rightarrow \mathrm{End} (V)$ and $F: \mathfrak{g} \rightarrow V$ by
\begin{align*}
\chi (x, y) = [s(x), s(y)]_\mathfrak{e} - s [x, y]_\mathfrak{g}, \qquad \rho_x v = [s(x) , v]_\mathfrak{e} \quad \text{ and } \quad F(x) = (Us - s N)(x),
\end{align*}
for $x, y \in \mathfrak{g}$ and $v \in V$. It can be checked that the map $\rho$ doesn't depend on the choice of $s$. Moreover, since the Lie bracket of $V$ is trivial, it follows from (\ref{nc1}), (\ref{nc3}) that the map $\rho$ makes the triple $(V, \rho, S)$ into a Nijenhuis representation of the Nijenhuis Lie algebra $(\mathfrak{g}, [~,~]_\mathfrak{g}, N)$, called the {\em induced Nijenhuis representation} from the abelian extension. Further, the identities (\ref{nc2}) and (\ref{nc4}) imply that the pair $(\chi, F)$ is a $2$-cocycle of the Nijenhuis Lie algebra $(\mathfrak{g}, [~,~]_\mathfrak{g}, N)$ with coefficients in the Nijenhuis representation $(V, \rho, S)$.
%More generally, we have the following classification result.

\medskip

Let $(\mathfrak{g}, [~,~]_\mathfrak{g}, N)$ be a Nijenhuis Lie algebra and $(V, \rho, S)$ be a given Nijenhuis representation of it. We denote by $\mathcal{E}xt ((\mathfrak{g}, N); (V, S))$ the set of all isomorphism classes of abelian extensions of the Nijenhuis Lie algebra $(\mathfrak{g}, [~,~]_\mathfrak{g}, N)$ by the pair $(V, S)$ so that the induced Nijenhuis representation coincides with the prescribed one. While restricting ourselves to such abelian extensions, Theorem \ref{main-thm-sec3} gives rise to the following result.

\begin{thm}\label{thm-abelian}
Let $(\mathfrak{g}, [~,~]_\mathfrak{g}, N)$ be a Nijenhuis Lie algebra and $(V, \rho, S)$ be a Nijenhuis representation of it. Then there is a bijection
\begin{align*}
\mathcal{E}xt ((\mathfrak{g}, N); (V, S)) ~ \! \cong ~ \! H^2 ((\mathfrak{g}, N) ; (V, S)).
\end{align*}
\end{thm}

\begin{remark}
    In \cite{agore} Agore and Militaru considered the problem of the extending structure (in short, the ES problem) as a unification of the extension problem and the factorization theorem. The description of extending structures was given by extending datums and the classification of the extending structures problem was answered by a suitable cohomological object. In \cite{peng-zhang} Peng and Zhang recently studied the extending structures problem in the context of Rota-Baxter Lie algebras and introduced the corresponding extending datums and the cohomological object. Since a Rota-Baxter Lie algebra and a Nijenhuis Lie algebra both can be regarded as generalizations of Lie algebra, it is natural to consider the problem of the extending structure for Nijenhuis Lie algebras. We will discuss this more general question and find its applications to the factorization problem in a subsequent work.
\end{remark}

\medskip

\section{Inducibility of a pair of Nijenhuis Lie algebra automorphisms}\label{sec4}
This section studies the inducibility problem of a pair of Nijenhuis Lie algebra automorphisms in a given non-abelian extension. Among others, we define the Wells map and construct the Wells exact sequence in this context. Finally, we end this section by considering the same inducibility problem in a given abelian extension of a Nijenhuis Lie algebra by a Nijenhuis representation.

\medskip

Let $(\mathfrak{g}, [~,~]_\mathfrak{g}, N)$ and $(\mathfrak{h}, [~,~]_\mathfrak{h}, S)$ be two Nijenhuis Lie algebras. Suppose
\begin{align}\label{ext-ind-aut}
\xymatrix{
0 \ar[r] & (\mathfrak{h}, [~,~]_\mathfrak{h}, S) \ar[r]^i & (\mathfrak{e}, [~,~]_\mathfrak{e}, U) \ar[r]^p & (\mathfrak{g}, [~,~]_\mathfrak{g}, N) \ar[r] & 0
}
\end{align}
is a given non-abelian extension of Nijenhuis Lie algebras. Let $\mathrm{Aut}_\mathfrak{h} (\mathfrak{e}, U)$ be the group of all Nijenhuis Lie algebra automorphisms $\gamma \in \mathrm{Aut} (\mathfrak{e}, U)$ for which $\mathfrak{h}$ is an invariant subspace (i.e. $\gamma (\mathfrak{h}) \subset \mathfrak{h}$). For $\gamma \in \mathrm{Aut}_\mathfrak{h} (\mathfrak{e}, U)$, we have $\gamma \big|_\mathfrak{h} \in \mathrm{Aut} (\mathfrak{h}, S)$. By choosing any section $s : \mathfrak{g} \rightarrow \mathfrak{e}$ of the map $p$, we also define a map $\overline{\gamma} : \mathfrak{g} \rightarrow \mathfrak{g}$ by $\overline{\gamma} (x) = (p \gamma s) (x)$, for $x \in \mathfrak{g}$. 
It is easy to see that the map $\overline{\gamma}$ is independent of the choice of the section $s$. Further, $\overline{\gamma}$ is a bijection. For any $x, y \in \mathfrak{g}$, we also observe that
\begin{align*}
\overline{\gamma} ( [x, y]_\mathfrak{g}) = p \gamma (s [x,y]_\mathfrak{g}) =~& p \gamma ( [s(x), s(y)]_\mathfrak{e} - \chi (x,y) ) \\
=~& p \gamma ( [s(x), s(y)]_\mathfrak{e} ) \quad (\text{as } \gamma (\mathfrak{h}) \subset \mathfrak{h} \text{ and } p|_\mathfrak{h} = 0) \\
=~& [p \gamma s (x), p \gamma s (y)]_\mathfrak{g} 
= [ \overline{\gamma} (x) , \overline{\gamma} (y)]_\mathfrak{g}
\end{align*}
and
\begin{align}\label{commute-n}
(N \overline{\gamma} - \overline{\gamma} N) (x) =~& (N p \gamma s - p \gamma s N)(x) \\
=~& (p U \gamma s - p \gamma s N)(x) \quad (\text{as } Np = pU) \nonumber \\
=~& p \gamma (Us - sN)(x) \quad (\text{as } U \gamma = \gamma U) \nonumber \\
=~& 0 \quad (\text{as } (Us -sN)(x) \in \mathfrak{h}, ~ \gamma (\mathfrak{h}) \subset \mathfrak{h} \text{ and } p|_\mathfrak{h} = 0). \nonumber
\end{align}
This shows that $\overline{\gamma} \in \mathrm{Aut} (\mathfrak{g}, N)$ is a Nijenhuis Lie algebra automorphism. As a result, we obtain a group homomorphism
\begin{align*}
\tau : \mathrm{Aut}_\mathfrak{h} (\mathfrak{e}, U) \rightarrow \mathrm{Aut} (\mathfrak{h}, S) \times \mathrm{Aut} (\mathfrak{g}, N) ~~~~ \text{ given by } ~~~~ \tau (\gamma) = (\gamma \big|_\mathfrak{h}, \overline{\gamma}).
\end{align*}
In this case, the pair $ (\gamma \big|_\mathfrak{h}, \overline{\gamma})$ is said to be induced by $\gamma$. Keeping this in mind, we say that a pair $(\beta, \alpha) \in \mathrm{Aut} (\mathfrak{h}, S) \times \mathrm{Aut} (\mathfrak{g}, N)$ of Nijenhuis Lie algebra automorphisms is {\bf inducible} if it lies in the image of the map $\tau$, i.e. there must exists a Nijenhuis Lie algebra automorphism $\gamma \in \mathrm{Aut}_\mathfrak{h} (\mathfrak{e}, U)$ such that $\gamma \big|_\mathfrak{h} = \beta$ and $\overline{\gamma} = \alpha$.

Our primary aim in this section is to find a necessary and sufficient condition for the inducibility of a given pair $(\beta, \alpha) \in \mathrm{Aut} (\mathfrak{h}, S) \times \mathrm{Aut} (\mathfrak{g}, N)$ of Nijenhuis Lie algebra automorphisms. We begin with the following result.

\begin{prop}\label{first-prop}
Let (\ref{ext-ind-aut}) be a non-abelian extension of Nijenhuis Lie algebras. For a section $s$, suppose the extension corresponds to the non-abelian $2$-cocycle $(\chi, \psi, F)$. Then a pair $(\beta, \alpha) \in \mathrm{Aut} (\mathfrak{h}, S) \times \mathrm{Aut} (\mathfrak{g}, N)$ of Nijenhuis Lie algebra automorphisms is inducible if and only if there exists a linear map $\lambda: \mathfrak{g} \rightarrow \mathfrak{h}$ satisfying the following conditions:
\begin{align}
\beta (\psi_x h) - \psi_{\alpha (x) } \beta (h) =~& [\lambda (x), \beta (h)]_\mathfrak{h}, \label{ab1}\\
\beta (\chi (x, y) ) - \chi (\alpha (x), \alpha (y)) =~& \psi_{\alpha (x)} \lambda (y) - \psi_{\alpha (y)} \lambda (x) - \lambda ([x, y]_\mathfrak{g}) + [ \lambda (x), \lambda (y)]_\mathfrak{h}, \label{ab2}\\
\beta (F (x)) - F (\alpha (x)) =~& S (\lambda (x)) - \lambda ( N (x)), \text{ for all } x, y \in \mathfrak{g} \text{ and } h \in \mathfrak{h}. \label{ab3}
\end{align}
\end{prop}

\begin{proof}
Let $(\beta, \alpha)$ be an inducible pair. That is, there exists a Nijenhuis Lie algebra automorphism $\gamma \in \mathrm{Aut}_\mathfrak{h} (\mathfrak{e}, U)$ such that $\gamma \big|_\mathfrak{h} = \beta$ and $p \gamma s = \alpha$. For any $ x \in \mathfrak{g}$, we observe that $(\gamma s - s \alpha) (x) \in \mathrm{ker} (p) $ which in turn implies that $(\gamma s - s \alpha) (x) \in \mathfrak{h}$. We define a map $\lambda : \mathfrak{g} \rightarrow \mathfrak{h}$ by
\begin{align*}
\lambda (x) := (\gamma s - s \alpha) (x), \text{ for } x \in \mathfrak{g}.
\end{align*} 
Then we observe that
\begin{align*}
\beta (\psi_x h) - \psi_{\alpha (x)} \beta (h) =~& \beta ( [s(x), h]_\mathfrak{e}) - [s \alpha (x), \beta (h)]_\mathfrak{e} \\
=~& [\gamma s (x), \gamma (h)]_\mathfrak{e} - [s \alpha (x), \beta (h)]_\mathfrak{e} \quad (\text{as } \beta = \gamma|_\mathfrak{h})\\
=~& [\gamma s (x), \beta (h)]_\mathfrak{e} - [s \alpha (x), \beta (h)]_\mathfrak{e}
= [\lambda (x), \beta (h)]_\mathfrak{h},
\end{align*}
for $x \in \mathfrak{g}$ and $h \in \mathfrak{h}$. Similarly, for any $x, y \in \mathfrak{g}$, we get that
\begin{align*}
&\psi_{\alpha (x)} \lambda (y) - \psi_{\alpha (y)} \lambda (x) - \lambda ([x, y]_\mathfrak{g}) + [ \lambda (x), \lambda (y)]_\mathfrak{h} \\
&= [s \alpha (x) , (\gamma s - s \alpha)(y)]_\mathfrak{e} - [s \alpha (y), (\gamma s - s \alpha)(x)]_\mathfrak{e} - (\gamma s - s \alpha)([x, y]_\mathfrak{g}) + [(\gamma s - s \alpha)(x), (\gamma s - s \alpha)(y)]_\mathfrak{h} \\
&= [s \alpha (x), \gamma s (y)]_\mathfrak{e} - [s \alpha (x) , s \alpha (y)]_\mathfrak{e} - [s \alpha (y), \gamma s (x)]_\mathfrak{e} + [s \alpha (y), s \alpha(x)]_\mathfrak{e} - \gamma s ([x, y]_\mathfrak{g}) + s [\alpha (x), \alpha(y)]_\mathfrak{g} \\
& \qquad + [\gamma s (x), \gamma s (y)]_\mathfrak{e} - [\gamma s (x), s \alpha (y)]_\mathfrak{e} - [s \alpha (x), \gamma s (y)]_\mathfrak{e} + [s \alpha (x), s \alpha(y)]_\mathfrak{e} \\
&= \big( [\gamma s (x), \gamma s (y)]_\mathfrak{e} - \gamma s ([x, y]_\mathfrak{g}) \big) + [s \alpha (y), s \alpha (x)]_\mathfrak{e} + s [\alpha(x), \alpha (y)]_\mathfrak{g} \quad (\text{after cancellation and rearranging}) \\
&= \gamma \big(  [s(x), s(y)]_\mathfrak{e} - s [x, y]_\mathfrak{g} \big)  - \big( [s\alpha (x) , s \alpha(y)]_\mathfrak{e} - s [\alpha (x), \alpha(y)]_\mathfrak{g}  \big) \\
&= \beta \big(  [s(x), s(y)]_\mathfrak{e} - s [x, y]_\mathfrak{g} \big) - \big( [s\alpha (x) , s \alpha(y)]_\mathfrak{e} - s [\alpha (x), \alpha(y)]_\mathfrak{g}  \big) \\
&= \beta (\chi (x, y)) - \chi (\alpha (x), \alpha (y))
\end{align*}
and also
\begin{align*}
\beta (F(x)) - F (\alpha (x)) 
&= \beta \big( U(s(x)) - s (N(x)) \big) - \big( U (s \alpha (x)) - s (N \alpha (x)) \big) \\
&= \gamma \big( U(s(x)) - s (N(x)) \big) -  U (s \alpha (x)) + s (N \alpha (x)) \quad (\text{as } \beta = \gamma|_\mathfrak{h}) \\
&= U \gamma (s(x)) - \gamma s (N(x)) - U (s \alpha (x)) + s ( \alpha N (x))  \quad (\text{as } U \gamma = \gamma U \text{ and } N \alpha = \alpha N) \\
&= U ((\gamma s - s \alpha)(x)) - (\gamma s - s \alpha)(N(x)) \\
&=U(\lambda(x)) - \lambda (N(x)) = S (\lambda (x)) - \lambda (N (x)).
\end{align*}
Hence the one direction of the proof follows. Conversely, suppose a linear map $\lambda : \mathfrak{g} \rightarrow \mathfrak{h}$ exists which satisfy the identities (\ref{ab1}), (\ref{ab2}) and (\ref{ab3}). First observe that, using the section $s$, the space $\mathfrak{e}$ can be identified with $s (\mathfrak{g}) \oplus \mathfrak{h}$. Equivalently, any element $e \in \mathfrak{e}$ can be uniquely written as $e = s(x) + h$, for some $x \in \mathfrak{g}$ and $h \in \mathfrak{h}$. Using this identification, we define a map $\gamma : \mathfrak{e} \rightarrow \mathfrak{e}$ by
\begin{align*}
\gamma (e) = \gamma ( s(x) + h) = s (\alpha (x)) + (\beta (h) + \lambda (x)), \text{ for } e = s(x) + h \in \mathfrak{e}.
\end{align*}
Since $s, \alpha, \beta$ are all injections, we can easily verify that $\gamma$ is so. The maps $\alpha$ and $\beta$ are also surjections which imply that $\gamma$ is also a surjection. Thus, $\gamma$ becomes a bijective map. By using the conditions (\ref{ab1}) and (\ref{ab2}), it has been shown in \cite{fre,das-hazra-mishra} that $\gamma : \mathfrak{e} \rightarrow \mathfrak{e}$ is a Lie algebra homomorphism. Moreover, for any $e = s(x) + h \in \mathfrak{e}$, we observe that
\begin{align}\label{compos}
(\gamma \circ U) (e)
&= (\gamma \circ U) ( s(x) + h) \\
&= \gamma \big(   Us (x)  + S(h)  \big) \nonumber \\
&= \gamma \big(  s N(x) + S(h) + F (x) \big) \quad (\text{since } F = Us -sN) \nonumber \\
&= s (\alpha N(x) ) + \beta (S(h) + F (x)) + \lambda (N(x)) \nonumber \\
&=  s (N \alpha (x)) + \beta S (h) + \beta F (x) + \lambda (N(x)) \nonumber  \\
&=  U (s \alpha (x)) - F (\alpha (x))  + \beta S (h) + \beta F (x) + \lambda (N(x)) \qquad (\text{since } F = Us -sN) \nonumber \\
&= U (s \alpha (x)) + S \beta (h) + \big( \beta F (x) + \lambda (N(x)) - F (\alpha (x)) \big) \qquad (\text{since } \beta S = S \beta) \nonumber  \\
&=  U (s \alpha (x)) + S \beta (h) + S (\lambda (x)) \qquad (\text{by } (\ref{ab3})) \nonumber  \\
&= U \big(  s (\alpha (x))+  \beta (h) + \lambda (x)  \big) \qquad (\text{as } U|_\mathfrak{h} = S) \nonumber  \\
&= (U \circ \gamma) (e). \nonumber 
\end{align}
Finally, it follows from the definition of $\gamma$ that $\gamma (\mathfrak{h}) \subset \mathfrak{h}$. Thus, $\gamma \in \mathrm{Aut}_\mathfrak{h} (\mathfrak{e}, U)$. Additionally, it is easy to see that $\gamma \big|_\mathfrak{h} = \beta$ and $\overline{\gamma } = p \gamma s = \alpha$. This shows that the pair $(\beta, \alpha)$ is inducible.
\end{proof}

It is important to observe that the necessary and sufficient condition obtained in the previous proposition depends on the section $s$. However, the best result would be to find a condition which is independent of any section. For this, we will adapt the method of Wells (who developed for abstract groups) \cite{wells} in the context of Nijenhuis Lie algebras.

For a fixed section $s$, suppose the given non-abelian extension corresponds to the non-abelian $2$-cocycle $(\chi, \psi, F)$. Let $(\beta, \alpha) \in \mathrm{Aut} (\mathfrak{h}, S) \times \mathrm{Aut} (\mathfrak{g}, N)$ be a given pair of Nijenhuis Lie algebra automorphisms. We define a triple $(\chi_{(\beta, \alpha)}, \psi_{(\beta, \alpha)}, F_{(\beta, \alpha)})$ of linear maps $\chi_{(\beta, \alpha)} : \wedge^2 \mathfrak{g} \rightarrow \mathfrak{h}$, $\psi_{(\beta, \alpha)} : \mathfrak{g} \rightarrow \mathrm{Der} (\mathfrak{h})$ and $F_{(\beta, \alpha)} : \mathfrak{g} \rightarrow \mathfrak{h}$ by
\medskip
\begin{align*}
\chi_{(\beta, \alpha)} (x, y) := \beta \chi (\alpha^{-1} (x), \alpha^{-1} (y)), \quad (\psi_{(\beta, \alpha)})_x h := \beta ( \psi_{\alpha^{-1} (x)} \beta^{-1} (h)) ~~~~ \text{ and } ~~~~ F_{(\beta, \alpha)} (x) := \beta ( F (\alpha^{-1} (x))),
\end{align*}
\medskip 
for $x, y \in \mathfrak{g}$ and $h \in \mathfrak{h}$. Then we have the following result.

\begin{prop}
The triple $(\chi_{(\beta, \alpha)}, \psi_{(\beta, \alpha)}, F_{(\beta, \alpha)})$ is also a non-abelian $2$-cocycle.
\end{prop}

\begin{proof}
Since $(\chi, \psi, F)$ is a non-abelian $2$-cocycle, the identities in (\ref{nc1})-(\ref{nc4}) are hold. In this identities, if we replace $x, y, z , h$ respectively by $\alpha^{-1} (x), \alpha^{-1} (y), \alpha^{-1} (z) , \beta^{-1} (h)$ and use the definitions of $\chi_{(\beta, \alpha)}$, $\psi_{(\beta, \alpha)}$ and $F_{(\beta, \alpha)}$, we obtain the corresponding non-abelian $2$-cocycle identities for the triple $(\chi_{(\beta, \alpha)}, \psi_{(\beta, \alpha)}, F_{(\beta, \alpha)})$. This shows the desired result.
\end{proof}

Note that the non-abelian $2$-cocycle $(\chi_{(\beta, \alpha)}, \psi_{(\beta, \alpha)}, F_{(\beta, \alpha)})$ obtained in the above proposition also depends on the section $s$. However, we have the following important result.

\begin{prop}\label{prop-a}
The equivalence class of the non-abelian $2$-cocycle $(\chi_{(\beta, \alpha)}, \psi_{(\beta, \alpha)}, F_{(\beta, \alpha)}) - (\chi, \psi, F)$ doesn't depend on the choice of the section $s$.
\end{prop}

\begin{proof}
Let $\widetilde{s}$ be any other section of the map $p$. If $(\widetilde{\chi}, \widetilde{\psi}, \widetilde{F})$ is the corresponding non-abelian $2$-cocycle produced from the same extension then we have seen earlier that the non-abelian $2$-cocycles $(\chi, \psi, F)$ and $(\widetilde{\chi}, \widetilde{\psi}, \widetilde{F})$ are equivalent, and an equivalence is given by the map $\varphi:= s - \widetilde{s}$. Moreover, for any $x \in \mathfrak{g}$ and $h \in \mathfrak{h}$, we observe that
\begin{align*}
(\psi_{(\beta, \alpha)})_x (h) - (\widetilde{\psi}_{(\beta, \alpha)})_x (h) =~& \beta (\psi_{\alpha^{-1} (x)} \beta^{-1}(h)) - \beta (\widetilde{\psi}_{\alpha^{-1} (x)} \beta^{-1}(h)) \\
=~& \beta [  \varphi \alpha^{-1} (x), \beta^{-1}(h)]_\mathfrak{h} = [\beta \varphi \alpha^{-1} (x), h]_\mathfrak{h}.
\end{align*}
Similarly, by straightforward calculations, we have
\begin{align*}
\chi_{(\beta, \alpha)} (x,y) - \widetilde{\chi}_{(\beta, \alpha)} (x,y) =~& (\widetilde{\psi}_{(\beta, \alpha)})_x (\beta \varphi \alpha^{-1} (y)) - (\widetilde{\psi}_{(\beta, \alpha)})_y (\beta \varphi \alpha^{-1} (x)) \\
~& \qquad \qquad \qquad - \beta \varphi \alpha^{-1} ([x,y]_\mathfrak{g}) + [\beta \varphi \alpha^{-1} (x), \beta \varphi \alpha^{-1} (y)]_\mathfrak{h},
\end{align*}
\begin{align*}
F_{(\beta, \alpha)} (x) - \widetilde{F}_{(\beta, \alpha)} (x) = S \big( \beta \varphi \alpha^{-1} (x)  \big) - \beta \varphi \alpha^{-1} (N(x)),
\end{align*}
for $x, y \in \mathfrak{g}$. This shows that $(\chi_{(\beta, \alpha)}, \psi_{(\beta, \alpha)}, F_{(\beta, \alpha)})$ and $( \widetilde{\chi}_{(\beta, \alpha)}, \widetilde{\psi}_{(\beta, \alpha)}, \widetilde{F}_{(\beta, \alpha)})$ are also equivalent by the map $\beta \varphi \alpha^{-1}$. As a result, $(\chi_{(\beta, \alpha)}, \psi_{(\beta, \alpha)}, F_{(\beta, \alpha)}) - (\chi, \psi, F)$ and $( \widetilde{\chi}_{(\beta, \alpha)}, \widetilde{\psi}_{(\beta, \alpha)}, \widetilde{F}_{(\beta, \alpha)}) - (\widetilde{\chi}, \widetilde{\psi}, \widetilde{F})$ are equivalent by the map $\beta \varphi \alpha^{-1} - \varphi$. Hence the result follows.
\end{proof}

With the notations from the above discussions, we now define the Wells map 
\begin{align*}
\mathcal{W} : \mathrm{Aut} (\mathfrak{h}, S) \times \mathrm{Aut} (\mathfrak{g}, N) \rightarrow H^2_{nab} ((\mathfrak{g}, N); (\mathfrak{h}, S)) ~~~ \text{ by } ~~~~ \mathcal{W} (\beta, \alpha) := [  (\chi_{(\beta, \alpha)}, \psi_{(\beta, \alpha)}, F_{(\beta, \alpha)}) - (\chi, \psi, F) ],
\end{align*}
for $(\beta, \alpha) \in  \mathrm{Aut} (\mathfrak{h}, S) \times \mathrm{Aut} (\mathfrak{g}, N)$. It follows from Proposition \ref{prop-a} that the Wells map is independent of the choice of any section. Using this Wells map, we now obtain the following result for the inducibility of a pair of Nijenhuis Lie algebra automorphisms.

\begin{thm}\label{thm-ind-aut}
Let (\ref{ext-ind-aut}) be a non-abelian extension of the Nijenhuis Lie algebra $(\mathfrak{g}, [~,~]_\mathfrak{g}, N)$ by another Nijenhuis Lie algebra $(\mathfrak{h}, [~,~]_\mathfrak{h}, S)$. Then a pair $(\beta, \alpha) \in \mathrm{Aut} (\mathfrak{h}, S) \times \mathrm{Aut} (\mathfrak{g}, N)$ of Nijenhuis Lie algebra automorphisms is inducible if and only if $\mathcal{W} (\beta, \alpha) = 0$.
\end{thm} 

\begin{proof}
Let $(\beta, \alpha) \in \mathrm{Aut} (\mathfrak{h}, S) \times \mathrm{Aut} (\mathfrak{g}, N)$ be an inducible pair of Nijenhuis Lie algebra automorphisms. For any fixed section $s$, let the given non-abelian extension produce the non-abelian $2$-cocycle $(\chi, \psi, F)$. Then by Proposition \ref{first-prop}, there exists a linear map $\lambda : \mathfrak{g} \rightarrow \mathfrak{h}$ satisfying (\ref{ab1})-(\ref{ab3}). In these identities, if we replace $x, y, h$ respectively by $\alpha^{-1} (x), \alpha^{-1} (y), \beta^{-1} (h)$, we get that
\begin{align*}
(\psi_{(\beta, \alpha)})_x h - \psi_x h =~& [\lambda \alpha^{-1} (x), h]_\mathfrak{h}, \\
\chi_{(\beta, \alpha)} (x,y) - \chi (x,y) =~& \psi_x \lambda \alpha^{-1} (y) - \psi_y \lambda \alpha^{-1} (x) - \lambda \alpha^{-1} ([x,y]_\mathfrak{g}) + [\lambda \alpha^{-1} (x), \lambda \alpha^{-1}(y)]_\mathfrak{h},\\
F_{(\beta, \alpha)} (x) - F (x) =~& S (\lambda \alpha^{-1} (x)) - \lambda \alpha^{-1} (N(x)).
\end{align*}
This shows that the non-abelian $2$-cocycles $(\chi_{(\beta, \alpha)}, \psi_{(\beta, \alpha)}, F_{(\beta, \alpha)})$ and $(\chi, \psi, F)$ are equivalent by the map $\lambda \alpha^{-1}$. This in turn implies that the cohomology class $\mathcal{W} (\beta, \alpha) = [ (\chi_{(\beta, \alpha)}, \psi_{(\beta, \alpha)}, F_{(\beta, \alpha)}) - (\chi, \psi, F)] $ is trivial.

Conversely, assume that $\mathcal{W} (\beta, \alpha) = 0$. As before, let $s$ be any section and $(\chi, \psi, F)$ be the non-abelian $2$-cocycle produced from the given non-abelian extension. Then it follows that the non-abelian $2$-cocycles $(\chi_{(\beta, \alpha)}, \psi_{(\beta, \alpha)}, F_{(\beta, \alpha)})$ and $(\chi, \psi, F)$ are equivalent (say by the map $\varphi : \mathfrak{g} \rightarrow \mathfrak{h}$). Then it is easy to see that the map $\lambda := \varphi \alpha : \mathfrak{g} \rightarrow \mathfrak{h}$ satisfy the identities (\ref{ab1})-(\ref{ab3}). Therefore, by Proposition \ref{first-prop}, the pair $(\beta, \alpha)$ is inducible.
\end{proof}

It follows from the above theorem that $\mathcal{W} (\beta, \alpha)$ is an obstruction for the inducibility of the pair $(\beta, \alpha)$. This generalizes the well-known fact (from the contexts of abstract groups, Lie algebras, Rota-Baxter Lie algebras etc.) that the images of the Wells map are the obstructions for the inducibility of the pair of automorphisms. In the following, we will show that the Wells map defined above fits into a short exact sequence.

\begin{thm}\label{thm-wells-ses-aut}
Let $0 \rightarrow (\mathfrak{h}, [~,~]_\mathfrak{h}, S) \xrightarrow{i} (\mathfrak{e}, [~,~]_\mathfrak{e}, U) \xrightarrow{p} (\mathfrak{g}, [~,~]_\mathfrak{g}, N) \rightarrow 0$ be a non-abelian extension of Nijenhuis Lie algebras. Then there is an exact sequence
\begin{align}\label{wells-ses}
1 \rightarrow \mathrm{Aut}_{\mathfrak{h}, \mathfrak{g}} (\mathfrak{e}, U) \hookrightarrow \mathrm{Aut}_\mathfrak{h} (\mathfrak{e}, U) \xrightarrow{\tau} \mathrm{Aut} (\mathfrak{h}, S) \times \mathrm{Aut} (\mathfrak{g}, N) \xrightarrow{\mathcal{W}} H^2_{nab} ((\mathfrak{g}, N); (\mathfrak{h}, S)),
\end{align}
where $\mathrm{Aut}_{\mathfrak{h}, \mathfrak{g}} (\mathfrak{e}, U) = \{ \gamma \in \mathrm{Aut} (\mathfrak{e}, U) ~ \! | ~ \! \tau (\gamma) = (\mathrm{Id}_\mathfrak{h}, \mathrm{Id}_\mathfrak{g}) \}.$
\end{thm}

\begin{proof}
First, note that the sequence (\ref{wells-ses}) is exact at the first place as $\mathrm{Aut}_{\mathfrak{h}, \mathfrak{g}} (\mathfrak{e}, U) \hookrightarrow \mathrm{Aut}_\mathfrak{h} (\mathfrak{e}, U)$ is the inclusion map. To show that the sequence is exact at the second place, we start with an element $\gamma \in \mathrm{ker} (\tau)$. Then we have $\tau (\gamma) = (\mathrm{Id}_\mathfrak{h}, \mathrm{Id}_\mathfrak{g})$ which implies that $\gamma \in \mathrm{Aut}_{\mathfrak{h}, \mathfrak{g}} (\mathfrak{e}, U)$. Conversely, if $\gamma \in \mathrm{Aut}_{\mathfrak{h}, \mathfrak{g}} (\mathfrak{e}, U)$ then we obviously get that $\gamma \in \mathrm{ker} (\tau)$. Hence $\mathrm{ker}(\tau) = \mathrm{Aut}_{\mathfrak{h}, \mathfrak{g}} (\mathfrak{e}, U)$ and thus the sequence is exact at the second place. Finally, we take an element $(\beta, \alpha) \in \mathrm{ker} (\mathcal{W})$. That is, $\mathcal{W} (\beta, \alpha) = 0$ which implies that the pair $(\beta, \alpha)$ is inducible. Therefore, there exists a Nijenhuis Lie algebra automorphism $\gamma \in \mathrm{Aut}_\mathfrak{h} (\mathfrak{e}, U)$ such that $\tau (\gamma) = (\beta, \alpha)$ which shows that $(\beta, \alpha) \in \mathrm{im}(\tau)$. Conversely, if $(\beta, \alpha) \in \mathrm{im} (\tau)$ then $(\beta, \alpha)$ is inducible and hence $(\beta, \alpha) \in \mathrm{ker} (\mathcal{W})$. Thus, $\mathrm{ker} (\mathcal{W}) = \mathrm{im} (\tau)$ which verifies that the sequence is also exact at the third place. This completes the proof.
\end{proof}

In the following, we consider the inducibility problem and the Wells exact sequence in a given abelian extension. Let $(\mathfrak{g}, [~,~]_\mathfrak{g}, N)$ be a Nijenhuis Lie algebra and $(V,\rho, S)$ be a Nijenhuis representation of it. Suppose (\ref{abel-ext}) is an abelian extension of the Nijenhuis Lie algebra $(\mathfrak{g}, [~,~]_\mathfrak{g}, N)$ by the Nijenhuis representation $(V, \rho, S)$. That is, the induced Nijenhuis representation from the above abelian extension coincides with the prescribed one. We set
\begin{align*}
\mathcal{C}_\rho := \big\{ (\beta, \alpha) \in \mathrm{Aut} (V, S) \times \mathrm{Aut} (\mathfrak{g}, N) ~ \! | ~ \! \beta (\rho_x v) = \rho_{\alpha (x)} \beta (v), \forall x \in \mathfrak{g}, v \in V \big\}
\end{align*}
and call it the space of all {\em compatible pairs} of Nijenhuis Lie algebra automorphisms. Then $\mathcal{C}_\rho$ is a subgroup of $\mathrm{Aut} (V, S) \times \mathrm{Aut} (\mathfrak{g}, N)$. If $\gamma \in \mathrm{Aut}_V (\mathfrak{e}, U)$ then it can be easily checked that $\tau (\gamma) = (\gamma \big|_V, \overline{\gamma}) \in \mathcal{C}_\rho$.

Given any section $s$, let the above abelian extension produce the $2$-cocycle $(\chi, F)$. Then for any $(\beta, \alpha) \in \mathrm{Aut} (V, S) \times \mathrm{Aut} (\mathfrak{g}, N)$, the pair $(\chi_{(\beta, \alpha)} , F_{(\beta, \alpha)})$ need not be a $2$-cocycle. However, if $(\beta, \alpha) \in \mathcal{C}_\rho$ then it turns out that $(\chi_{(\beta, \alpha)} , F_{(\beta, \alpha)})$ is a $2$-cocycle. Hence the Wells map in this context is the map $\mathcal{W} : \mathcal{C}_\rho \rightarrow H^2 ((\mathfrak{g}, N); (V, S))$ given by
\begin{align}\label{wells-aut-abel}
\mathcal{W} (\beta, \alpha) := [(\chi_{(\beta, \alpha)} , F_{(\beta, \alpha)}) - (\chi, F) ], \text{ for } (\beta, \alpha) \in \mathcal{C}_\rho.
\end{align}
This map is independent of the choice of the section $s$. Hence, in this context, Theorems \ref{thm-ind-aut} and \ref{thm-wells-ses-aut} can be summarized as follows.

\begin{thm}
Let (\ref{abel-ext}) be an abelian extension of the Nijenhuis Lie algebra $(\mathfrak{g}, [~,~]_\mathfrak{g}, N)$ by the Nijenhuis representation $(V, \rho, S)$. 
\begin{itemize}
\item[(i)] Then a pair $(\beta, \alpha) \in \mathrm{Aut} (V, S) \times \mathrm{Aut} (\mathfrak{g}, N)$ of Nijenhuis Lie algebra automorphisms is inducible if and only if $(\beta, \alpha) \in \mathcal{C}_\rho$ and $\mathcal{W} (\beta, \alpha) = 0$.
\item[(ii)] There is an exact sequence 
\begin{align}\label{wells-ses-aut-abel}
1 \rightarrow \mathrm{Aut}_{V, \mathfrak{g}} (\mathfrak{e}, U) \hookrightarrow \mathrm{Aut}_V (\mathfrak{e}, U) \xrightarrow{\tau} \mathcal{C}_\rho \xrightarrow{\mathcal{W}} H^2 ((\mathfrak{g}, N); (V, S)).
\end{align}
\end{itemize}
\end{thm}

\medskip

An abelian extension (\ref{abel-ext}) of the Nijenhuis Lie algebra $(\mathfrak{g}, [~,~]_\mathfrak{g}, N)$ by a Nijenhuis representation $(V, \rho, S)$ is said to be {\em split} if there is a section $s: \mathfrak{g} \rightarrow \mathfrak{e}$ which is a morphism of Nijenhuis Lie algebras. In this case, we can identify the Nijenhuis Lie algebra $(\mathfrak{e}, [~,~]_\mathfrak{e}, U)$ with the semidirect product Nijenhuis Lie algebra $(\mathfrak{g} \oplus V, [~,~]_\ltimes, N \oplus S)$. Moreover, with this identification, the map $i$ is the obvious inclusion and $p$ is the projection onto the subspace $\mathfrak{g}$. With respect to the above section $s$, if the abelian extension produced the $2$-cocycle $(\chi, F)$ then we have
\begin{align*}
\chi (x, y) = [s(x), s(y)]_\mathfrak{e} - s [x, y]_\mathfrak{g} = 0 \quad \text{ and } \quad F(x) = (Us- s N) (x) = 0, \text{ for all } x, y \in \mathfrak{g} 
\end{align*}
as $s: \mathfrak{g} \rightarrow \mathfrak{e}$ is a morphism of Nijenhuis Lie algebras. Thus, the $2$-cocycle $(\chi, F)$ is trivial and so the Wells map (\ref{wells-aut-abel}) vanishes identically. Hence the Wells exact sequence (\ref{wells-ses-aut-abel}) takes the form
\begin{align}\label{takes-the-form}
1 \rightarrow \mathrm{Aut}_{V, \mathfrak{g}} (\mathfrak{e}, U) \hookrightarrow \mathrm{Aut}_V (\mathfrak{e}, U) \xrightarrow{\tau} \mathcal{C}_\rho \rightarrow 1.
\end{align}
Next, for any $(\beta, \alpha) \in \mathcal{C}_\rho$, we define a map $\gamma_{(\beta, \alpha)} : \mathfrak{e} \rightarrow \mathfrak{e}$ by $\gamma_{(\beta, \alpha)} ( s(x) + u) := s (\alpha (x)) + \beta (u)$, for any $s(x)+ u \in \mathfrak{e}$. Then it is easy to see that $\gamma_{(\beta, \alpha)} \in \mathrm{Aut}_V (\mathfrak{e}, U)$. Hence we obtain a map $t : \mathcal{C}_\rho \rightarrow \mathrm{Aut}_V (\mathfrak{e}, U)$ given by $t (\beta, \alpha) := \gamma_{(\beta, \alpha)}$, for $(\beta, \alpha) \in \mathcal{C}_\rho$. It is easy to see that $t$ is a group homomorphism and $\tau t = \mathrm{Id}_{\mathcal{C}_\rho}$ which implies that the sequence (\ref{takes-the-form}) is a split exact sequence of groups. Thus, when the abelian extension (\ref{abel-ext}) is split, we obtain the following isomorphism of groups
\begin{align*}
\mathrm{Aut}_V (\mathfrak{e}, U) ~ \! \cong ~ \! \mathcal{C}_\rho \ltimes \mathrm{Aut}_{V, \mathfrak{g}} (\mathfrak{e}, U).
\end{align*}

\medskip

\section{Inducibility of a pair of Nijenhuis Lie algebra derivations}\label{sec5}

In this section, we study the inducibility of a pair of Nijenhuis Lie algebra derivations in a given abelian extension. We find a necessary and sufficient condition for this inducibility problem in terms of a suitable Wells map. In the end, we also derive the corresponding Wells exact sequence in the present context.

\medskip

Let $(\mathfrak{g}, [~,~]_\mathfrak{g}, N)$ be a Nijenhuis Lie algebra. A linear map $D: \mathfrak{g} \rightarrow \mathfrak{g}$ is said to be a {\em Nijenhuis Lie algebra derivation} if $D$ is a usual derivation on the Lie algebra $(\mathfrak{g}, [~,~]_\mathfrak{g})$, i.e.
\begin{align*}
D ([x, y]_\mathfrak{g}) = [D(x), y]_\mathfrak{g} + [x, D(y)]_\mathfrak{g}, \text{ for } x, y \in \mathfrak{g}
\end{align*}
satisfying additionally $N \circ D= D \circ N$. We denote the space of all Nijenhuis Lie algebra derivations on $(\mathfrak{g}, [~,~]_\mathfrak{g}, N)$ simply by the notation $\mathrm{Der} (\mathfrak{g}, N)$. Then $\mathrm{Der} (\mathfrak{g}, N)$ is naturally a Lie algebra with the commutator bracket. More generally, let $(\mathfrak{g}, [~,~]_\mathfrak{g}, N)$ be a Nijenhuis Lie algebra and $(V, \rho, S)$ be a Nijenhuis representation of it. Then a linear map $d: \mathfrak{g} \rightarrow V$ is said to be a Nijenhuis Lie algebra derivation (with values in the Nijenhuis representation) if
\begin{align*}
S \circ d = d \circ N \quad \text{ and } \quad d ([x, y]_\mathfrak{g}) = \rho_x d(y) - \rho_y d(x), \text{ for all } x, y \in \mathfrak{g}. 
\end{align*}
We denote the space of such Nijenhuis Lie algebra derivations by $\mathrm{Der} ((\mathfrak{g}, N); (V, S)).$

%Note that Nijenhuis Lie algebra derivations can be also defined on a Nijenhuis Lie algebra with values in a Nijenhuis representation. We postpone the study of such Nijenhuis Lie algebra derivations for a separate work.

Let
\begin{align}\label{abel-last}
\xymatrix{
0 \ar[r] & (V, [~,~]_V = 0, S) \ar[r]^i & (\mathfrak{e}, [~,~]_\mathfrak{e}, U) \ar[r]^p & (\mathfrak{g}, [~,~]_\mathfrak{g}, N) \ar[r] & 0
}
\end{align}
be a fixed abelian extension of the Nijenhuis Lie algebra $(\mathfrak{g}, [~,~]_\mathfrak{g}, N)$ by a Nijenhuis representation $(V, \rho, S)$. Suppose $\mathrm{Der}_V (\mathfrak{e}, U) = \{ D \in \mathrm{Der} (\mathfrak{e}, U) ~ \! | ~ \! D (V) \subset V \}$ is the set of all Nijenhuis Lie algebra derivations on $(\mathfrak{e}, [~,~]_\mathfrak{g}, U)$ that make the subspace $V$ invariant. Then for any $D \in \mathrm{Der}_V (\mathfrak{e}, U)$, we have $D |_V \in \mathrm{Der} (V, S)$. By choosing a section $s: \mathfrak{g} \rightarrow \mathfrak{e}$ of the map $p$, we define a map $\overline{D} : \mathfrak{g} \rightarrow \mathfrak{g}$ by $\overline{D} (x) := (p D s) (x)$, for $x \in \mathfrak{g}$. The map $\overline{D}$ is independent of the choice of $s$. Moreover, for any $x, y \in \mathfrak{g}$, we observe that
\begin{align*}
\overline{D} ([x, y]_\mathfrak{g}) = pD s ( [x, y]_\mathfrak{g})
=~& p D ( [s(x), s(y)]_\mathfrak{e} - \chi (x, y))\\
=~& p D ( [s(x), s(y)]_\mathfrak{e} ) \quad (\text{as } ~ \! D (V) \subset V \text{ and } p |_V = 0) \\
=~& p ( [Ds(x), s(y)]_\mathfrak{e} + [s(x), Ds(y)]_\mathfrak{e} )\\
=~& [ \overline{D} (x), y]_\mathfrak{g} + [x, \overline{D} (y)]_\mathfrak{g} \quad (\text{as } ~ \! pDs = \overline{D} \text{ and } ps = \mathrm{Id}_\mathfrak{g})
\end{align*}
and similar to (\ref{commute-n}), one can show that $N \overline{D} = \overline{D} N$. This shows that $\overline{D} \in \mathrm{Der} (\mathfrak{g}, N)$ is a Nijenhuis Lie algebra derivation on $(\mathfrak{g}, [~,~]_\mathfrak{g}, N)$. Hence there is a well-defined map 
\begin{align*}
\eta : \mathrm{Der}_V (\mathfrak{e}, U) \rightarrow \mathrm{Der} (V, S) \times \mathrm{Der} (\mathfrak{g}, N) ~~~~ \text{ given by } ~~~~ \eta (D) = (D \big|_V, \overline{D}).
\end{align*}
A pair of Nijenhuis Lie algebra derivations $(D_V, D_\mathfrak{g}) \in \mathrm{Der} (V, S) \times \mathrm{Der} (\mathfrak{g}, N)$ is said to be {\bf inducible} if it lies in the image of the map $\eta$. Our main aim in this section is to find a necessary and sufficient condition for the inducibility of a given pair of Nijenhuis Lie algebra derivations. 

\begin{prop}\label{prop-ind-der}
Let $(D_V, D_\mathfrak{g}) \in \mathrm{Der} (V, S) \times \mathrm{Der} (\mathfrak{g}, N)$ be an inducible pair of Nijenhuis Lie algebra derivations. Then
\begin{align}\label{ind-der}
D_V (\rho_x v) = \rho_{D_\mathfrak{g} (x)} v + \rho_x D_V (v), \text{ for any } x \in \mathfrak{g} \text{ and } v \in V.
\end{align}
\end{prop}

\begin{proof}
Since $(D_V, D_\mathfrak{g})$ is an inducible pair, there exists a Nijenhuis Lie algebra derivation $D \in \mathrm{Der}_V (\mathfrak{e}, U)$ such that $D|_V = D_V$ and $pDs = D_\mathfrak{g}$, where $s$ is any arbitrary section. Hence we have 
\begin{align*}
D_V (\rho_x v) =~& D ( [s(x), v]_\mathfrak{e}) \\
=~& [D s(x) , v]_\mathfrak{e} + [s(x), D (v)]_\mathfrak{e} \\
=~& [s D_\mathfrak{g} (x), v]_\mathfrak{e} + [s(x),  D_V (v) ]_\mathfrak{e} \quad (\because ~ \! (Ds -s D_\mathfrak{g}) (x) \in \mathrm{ker} (p) = \mathrm{im} (i) \cong V \text{ and } D |_V = D_V ) \\
=~& \rho_{D_\mathfrak{g} (x)} v + \rho_x D_V (v).
\end{align*}
This completes the proof.
\end{proof}

Motivated by the above result, we define
\begin{align*}
\mathcal{D}_\rho := \{ (D_V, D_\mathfrak{g}) \in \mathrm{Der} (V, S) \times \mathrm{Der} (\mathfrak{g}, N) ~ \! | ~ \! \text{ the identity } (\ref{ind-der}) \text{ holds} \}.
\end{align*}
It is easy to see that $\mathcal{D}_\rho$ is a Lie algebra with the componentwise commutator bracket. Let $(D_V, D_\mathfrak{g}) \in \mathcal{D}_\rho$. For any linear maps $\chi : \wedge^2 \mathfrak{g} \rightarrow V$ and $F : \mathfrak{g} \rightarrow V$, we define new maps $\chi_{(D_V, D_\mathfrak{g})} : \wedge^2 \mathfrak{g} \rightarrow V$ and $F_{ (D_V, D_\mathfrak{g}) } : \mathfrak{g} \rightarrow V$ by 
\begin{align*}
\chi_{(D_V, D_\mathfrak{g})} (x, y) := D_V (\chi (x, y)) -\chi (D_\mathfrak{g}(x), y) - \chi (x, D_\mathfrak{g} (y)) ~~~~ \text{ and } ~~~~ F_{ (D_V, D_\mathfrak{g}) } (x) :=  D_V (F (x)) - F (D_\mathfrak{g} (x)),
\end{align*}
for $x , y \in \mathfrak{g}$. Then we have the following result.

\begin{prop}\label{propo-final}
\begin{itemize}
\item[(i)]  If $(\chi, F)$ is a $2$-cocycle of the Nijenhuis Lie algebra $(\mathfrak{g}, [~, ~ ]_\mathfrak{g}, N)$ with coefficients in the Nijenhuis representation $(V, \rho, S)$ then the pair $ ( \chi_{(D_V, D_\mathfrak{g})}, F_{ (D_V, D_\mathfrak{g}) }   )$ is so.
\item[(ii)] If $(\chi, F)$ and $(\chi', F')$ are two cohomologous $2$-cocycles then the $2$-cocycles $ ( \chi_{(D_V, D_\mathfrak{g})}, F_{ (D_V, D_\mathfrak{g}) }   )$ and $( \chi'_{(D_V, D_\mathfrak{g})}, F'_{ (D_V, D_\mathfrak{g}) }   )$ are also cohomologous.
\end{itemize}
\end{prop}

\begin{proof}
(i) Since $(\chi, \psi)$ is a $2$-cocycle, the identities (\ref{2co-abel1}) and (\ref{2co-abel2}) are hold. For any $x, y, z \in \mathfrak{g}$, we observe that
\begin{align*}
&\rho_x ~ \! \chi_{(D_V, D_\mathfrak{g})} (y, z) + c.p. + \chi_{(D_V, D_\mathfrak{g})} (x, [y, z]_\mathfrak{g}) + c.p. \\
&= \rho_x \big\{  D_V (\chi (y, z)) - \chi (D_\mathfrak{g}(y), z) - \chi (y, D_\mathfrak{g} (z)) \big\} + c. p. \\
& \quad \qquad + \big\{  D_V (\chi (x, [y, z]_\mathfrak{g})) - \chi (D_\mathfrak{g} (x), [y, z]_\mathfrak{g}) - \chi (x, D_\mathfrak{g} [y, z]_\mathfrak{g})  \big\} + c.p. \\
&= \big\{  D_V \rho_x \chi (y, z) - \rho_{D_\mathfrak{g} (x)} \chi (y, z) - \rho_x \chi (D_\mathfrak{g} (y), z) - \rho_x \chi (y, D_\mathfrak{g} (z))  \big\} + c.p. \\
& \quad \qquad + \big\{ D_V (\chi (x, [y, z]_\mathfrak{g})) - \chi (D_\mathfrak{g} (x), [y, z]_\mathfrak{g}) - \chi (x, [D_\mathfrak{g} (y), z]_\mathfrak{g}) - \chi (x, [y, D_\mathfrak{g} (z)]_\mathfrak{g})   \big\} + c.p. \\
&= 0 \quad (\text{as } \chi \text{ satisfies } (\ref{2co-abel1})).
\end{align*}
Here $c.p.$ stands for the cyclic permutations of the inputs $x, y, z$. In the same way, by using the definitions of $\chi_{(D_V, D_\mathfrak{g})}$ and $ F_{(D_V, D_\mathfrak{g})}$ and also by using the identity (\ref{2co-abel2}), one can easily show that
\begin{align*}
 & \chi_{(D_V, D_\mathfrak{g})} (N (x), N (y)) - S \big( \chi_{(D_V, D_\mathfrak{g})} ( N (x), y) + \chi_{(D_V, D_\mathfrak{g})} (x, N (y)) - S ( \chi_{(D_V, D_\mathfrak{g})} ( x, y) ) \big) \\
 & \quad - F_{(D_V, D_\mathfrak{g})} \big( [N (x), y]_\mathfrak{g} + [x, N (y)]_\mathfrak{g} - N [x, y]_\mathfrak{g}    \big)  + \rho_{N(x)} F_{(D_V, D_\mathfrak{g})} (y) - \rho_{N (y)} F_{(D_V, D_\mathfrak{g})} (x) \\ 
 & \qquad \qquad - S \big( \rho_x F_{(D_V, D_\mathfrak{g})} (y) - \rho_y F_{(D_V, D_\mathfrak{g})} (x) - F_{(D_V, D_\mathfrak{g})} [x, y]_\mathfrak{g} \big) = 0.
\end{align*}
This shows that $(\chi_{(D_V, D_\mathfrak{g})}, F_{(D_V, D_\mathfrak{g})})$ is a $2$-cocycle.

\medskip

(ii) Since $(\chi, F)$ and $(\chi', F')$ are cohomologous $2$-cocycles, there exists a linear map $\varphi : \mathfrak{g} \rightarrow V$ such that
\begin{align}\label{useful-eqn}
\chi (x, y) - \chi' (x, y) = \rho_x \varphi (y) - \rho_y \varphi (x) - \varphi ([x, y]_\mathfrak{g}) ~~~~ \text{ and } ~~~~ F(x) - F'(x) = S (\varphi (x)) - \varphi (N(x)),
\end{align}
for all $x, y \in \mathfrak{g}$. Hence we have
\begin{align*}
&\chi_{(D_V, D_\mathfrak{g})} - \chi'_{(D_V, D_\mathfrak{g})} \\
&= \big\{ D_V (\chi (x, y)) -\chi (D_\mathfrak{g} (x), y) - \chi (x, D_\mathfrak{g} (y)) \big\} - \big\{ D_V (\chi' (x, y)) -\chi' (D_\mathfrak{g} (x), y) - \chi' (x, D_\mathfrak{g} (y)) \big\} \\
&=  D_V \big( \rho_x \varphi (y) - \rho_y \varphi (x) - \varphi ([x, y]_\mathfrak{g})   \big)  - \big\{  \rho_{D_\mathfrak{g} (x)} \varphi (y) - \rho_y \varphi (D_\mathfrak{g} (x)) - \varphi ([D_\mathfrak{g} (x), y]_\mathfrak{g})   \big\}\\
& \qquad \qquad - \big\{ \rho_x \varphi (D_\mathfrak{g} (y)) - \rho_{D_\mathfrak{g} (y)} \varphi (x) - \varphi ([x, D_\mathfrak{g} (y) ]_\mathfrak{g})  \big\} \qquad (\text{by } (\ref{useful-eqn}))\\
&= \rho_x (D_V \circ \varphi - \varphi \circ D_\mathfrak{g})(y) - \rho_y (D_V \circ \varphi - \varphi \circ D_\mathfrak{g}) (x) - (D_V \circ \varphi - \varphi \circ D_\mathfrak{g}) ([x, y]_\mathfrak{g}) \\
& \qquad \qquad \qquad \qquad (\text{by } (\ref{ind-der}) \text{ and since } D_\mathfrak{g} \text{ is a derivation})
\end{align*}
and
\begin{align*}
&F_{(D_V, D_\mathfrak{g})} - F'_{(D_V, D_\mathfrak{g})} \\
&= \big\{ D_V (F(x)) - F (D_\mathfrak{g} (x))  \big\} - \big\{ D_V (F'(x)) - F' (D_\mathfrak{g} (x)) \big\} \\
&= D_V (S \varphi - \varphi N) (x) - (S \varphi - \varphi N) D_\mathfrak{g} (x) \quad (\text{by } (\ref{useful-eqn})) \\
&= S ( (D_V \circ \varphi - \varphi \circ D_\mathfrak{g}) (x) )  - (D_V \circ \varphi - \varphi \circ D_\mathfrak{g}) (N (x)) \qquad (\because ~ \! S D_V = D_V S \text{ and } N D_\mathfrak{g} = D_\mathfrak{g} N).
\end{align*}
This shows that the $2$-cocycles $(\chi_{(D_V, D_\mathfrak{g})}, F_{(D_V, D_\mathfrak{g})})$ and $(\chi'_{(D_V, D_\mathfrak{g})}, F'_{(D_V, D_\mathfrak{g})})$ are cohomologous by the map $D_V \circ \varphi - \varphi \circ D_\mathfrak{g}$. Hence the proof follows.
\end{proof}

It follows from the above proposition that there is a well-defined map
\begin{align*}
\Theta : \mathcal{D}_\rho \rightarrow \mathrm{End} \big( H^2 ((\mathfrak{g}, N) ; (V, S)) \big) ~~~~ \text{ given by }  ~~~~ \Theta (D_V, D_\mathfrak{g}) [(\chi, F)] := [ ( \chi_{(D_V, D_\mathfrak{g})}, F_{ (D_V, D_\mathfrak{g}) }   )],
\end{align*}
for $(D_V, D_\mathfrak{g}) \in \mathcal{D}_\rho$ and $[(\chi, F)] \in H^2 ((\mathfrak{g}, N) ; (V, S))$. Then we have the following.

\begin{prop}
The map $\Theta$ gives a representation of the Lie algebra $\mathcal{D}_\rho$ on the second cohomology group $H^2 ((\mathfrak{g}, N) ; (V, S))$.
\end{prop}

\begin{proof}
Let $(D_V, D_\mathfrak{g}), (D'_V, D'_\mathfrak{g}) \in \mathcal{D}_\rho$ and $[(\chi, F)] \in H^2 ((\mathfrak{g}, N); (V, S))$. Then we have
\begin{align*}
&\big(  \Theta (D_V, D_\mathfrak{g}) \circ \Theta (D'_V, D'_\mathfrak{g}) -  \Theta (D'_V, D'_\mathfrak{g}) \circ \Theta (D_V, D_\mathfrak{g})  \big) [(\chi, F)] \\
&= \Theta (D_V, D_\mathfrak{g}) [ ( D'_V \circ \chi - \chi \circ (D'_\mathfrak{g} \otimes \mathrm{Id}_\mathfrak{g}) - \chi \circ (\mathrm{Id}_\mathfrak{g} \otimes D'_\mathfrak{g}) ~ \! , ~ \! D'_V \circ F - F \circ D'_\mathfrak{g}  )] \\
& \qquad \quad - \Theta (D'_V, D'_\mathfrak{g}) [ ( D_V \circ \chi - \chi \circ (D_\mathfrak{g} \otimes \mathrm{Id}_\mathfrak{g}) - \chi \circ (\mathrm{Id}_\mathfrak{g} \otimes D_\mathfrak{g}) ~ \! , ~ \! D_V \circ F - F \circ D_\mathfrak{g}  )] \\
&= \Big[ \Big( D_V \circ ( D'_V \circ \chi - \chi \circ (D'_\mathfrak{g} \otimes \mathrm{Id}_\mathfrak{g}) - \chi \circ (\mathrm{Id}_\mathfrak{g} \otimes D'_\mathfrak{g})  )  \\
& \qquad - ( D'_V \circ \chi - \chi \circ (D'_\mathfrak{g} \otimes \mathrm{Id}_\mathfrak{g}) - \chi \circ (\mathrm{Id}_\mathfrak{g} \otimes D'_\mathfrak{g})) \circ (D_\mathfrak{g} \otimes \mathrm{Id}_\mathfrak{g})  \\
& \qquad - ( D'_V \circ \chi - \chi \circ (D'_\mathfrak{g} \otimes \mathrm{Id}_\mathfrak{g}) - \chi \circ (\mathrm{Id}_\mathfrak{g} \otimes D'_\mathfrak{g})) \circ (\mathrm{Id}_\mathfrak{g} \otimes D_\mathfrak{g}) ~ \!, \\
& \qquad \qquad D_V \circ (D'_V \circ F - F \circ D'_\mathfrak{g}) - (D'_V \circ F - F \circ D'_\mathfrak{g}) \circ D_\mathfrak{g} \Big) \Big] \\
&- \Big[ \Big( D'_V \circ ( D_V \circ \chi - \chi \circ (D_\mathfrak{g} \otimes \mathrm{Id}_\mathfrak{g}) - \chi \circ (\mathrm{Id}_\mathfrak{g} \otimes D_\mathfrak{g})  )  \\
& \qquad - ( D_V \circ \chi - \chi \circ (D_\mathfrak{g} \otimes \mathrm{Id}_\mathfrak{g}) - \chi \circ (\mathrm{Id}_\mathfrak{g} \otimes D_\mathfrak{g})) \circ (D'_\mathfrak{g} \otimes \mathrm{Id}_\mathfrak{g})  \\
& \qquad - ( D_V \circ \chi - \chi \circ (D_\mathfrak{g} \otimes \mathrm{Id}_\mathfrak{g}) - \chi \circ (\mathrm{Id}_\mathfrak{g} \otimes D_\mathfrak{g})) \circ (\mathrm{Id}_\mathfrak{g} \otimes D'_\mathfrak{g}) ~ \!, \\
& \qquad \qquad D'_V \circ (D_V \circ F - F \circ D_\mathfrak{g}) - (D_V \circ F - F \circ D_\mathfrak{g}) \circ D'_\mathfrak{g} \Big) \Big] \\
& = \Big[ ( D_V \circ D'_V - D'_V \circ D_V ) \circ \chi - \chi \circ (  (D_\mathfrak{g} \circ D'_\mathfrak{g} - D'_\mathfrak{g} \circ D_\mathfrak{g}) \otimes \mathrm{Id}_\mathfrak{g}) - \chi \circ ( \mathrm{Id}_\mathfrak{g} \otimes (D_\mathfrak{g} \circ D'_\mathfrak{g} - D'_\mathfrak{g} \circ D_\mathfrak{g})) ~ \! ,\\
& \qquad \qquad \qquad \qquad (D_V \circ D'_V - D'_V \circ D_V ) \circ F - F \circ (D_\mathfrak{g} \circ D'_\mathfrak{g} - D'_\mathfrak{g} \circ D_\mathfrak{g})     \Big] \\
& = \Theta \big( D_V \circ D'_V - D'_V \circ D_V ~ \!, ~ \! D_\mathfrak{g} \circ D'_\mathfrak{g} - D'_\mathfrak{g} \circ D_\mathfrak{g} \big) [(\chi, F)] \\
& = \Theta \big(  [ (D_V, D_\mathfrak{g}), (D'_V, D'_\mathfrak{g})] \big) [(\chi, F)].
\end{align*}
This proves the desired result.
\end{proof}

For a fixed section $s$, let the given abelian extension (\ref{abel-last}) produced the $2$-cocycle $(\chi, F).$ Then we define a map $\mathcal{W} : \mathcal{D}_\rho \rightarrow H^2 ((\mathfrak{g}, N) ; (V, S))$ by
\begin{align*}
\mathcal{W} (D_V, D_\mathfrak{g}) := \Theta (D_V, D_\mathfrak{g}) [(\chi, F)] = [ ( \chi_{(D_V, D_\mathfrak{g})}, F_{ (D_V, D_\mathfrak{g}) }   )],
\end{align*}
for $(D_V, D_\mathfrak{g}) \in \mathcal{D}_\rho$. This is called the {\bf Wells map} in the context of Nijenhuis Lie algebra derivations.
Then it follows from Proposition \ref{propo-final} (ii) that the Wells map is independent of the choice of the section $s$. 
%The map $\mathcal{W}$ constructed above is called the {\bf Wells map}.
We are now ready to prove the main result of this section.

\begin{thm}\label{thm-ind-der} Let (\ref{abel-last}) be an abelian extension of the Nijenhuis Lie algebra $(\mathfrak{g}, [~,~]_\mathfrak{g}, N)$ by the Nijenhuis representation $(V, \rho, S)$.
Then a pair $(D_V, D_\mathfrak{g}) \in \mathrm{Der} (V, S) \times \mathrm{Der} (\mathfrak{g}, N)$ of Nijenhuis Lie algebra derivations is inducible if and only if $(D_V, D_\mathfrak{g}) \in \mathcal{D}_\rho$ and $ \mathcal{W} (D_V, D_\mathfrak{g}) = 0.$
\end{thm}

\begin{proof}
Let $(D_V, D_\mathfrak{g}) \in \mathrm{Der} (V, S) \times \mathrm{Der} (\mathfrak{g}, N)$ be an inducible pair of Nijenhuis Lie algebra derivations. Then there is a Nijenhuis Lie algebra derivation $D \in \mathrm{Der}_V (\mathfrak{e}, U)$ such that $D \big|_V = D_V$ and $pDs = D_\mathfrak{g}$, where $s$ is some section. Using this, we have seen in Proposition \ref{prop-ind-der} that $(D_V, D_\mathfrak{g}) \in \mathcal{D}_\rho$. To the section $s$, let the given abelian extension produce the $2$-cocycle $(\chi, F)$. Then we have
\begin{align*}
&\chi_{ (D_V, D_\mathfrak{g})} (x, y) \\
&= D_V (\chi (x, y)) - \chi (D_\mathfrak{g} (x), y) -\chi (x, D_\mathfrak{g} (y) ) \\
&=D_V \big(  [s(x), s(y)]_\mathfrak{e} - s [x, y]_\mathfrak{g} \big) - [s D_\mathfrak{g} (x), s(y)]_\mathfrak{e} + s [D_\mathfrak{g} (x), y]_\mathfrak{g} - [s(x), s D_\mathfrak{g} (y)]_\mathfrak{e} + s [x, D_\mathfrak{g} (y)]_\mathfrak{g} \\
&= [s(x) , Ds (y)]_\mathfrak{e} + [ Ds (x), s(y)]_\mathfrak{e} - Ds [x, y]_\mathfrak{g} - [s D_\mathfrak{g} (x), s(y)]_\mathfrak{e}
- [s(x), s D_\mathfrak{g} (y)]_\mathfrak{e} + s D_\mathfrak{g} [x, y]_\mathfrak{g} \\ & \qquad \qquad \qquad \qquad \qquad \qquad \qquad \qquad \qquad (\text{as } D_V = D \big|_V \text{ and } D, D_\mathfrak{g} \text{ are derivations}) \\
&= \rho_x (Ds - s D_\mathfrak{g}) (y)- \rho_y (Ds - s D_\mathfrak{g}) (x) - (Ds - s D_\mathfrak{g}) ([x, y]_\mathfrak{g}) \quad (\text{by rearrangements})
\end{align*}
and
\begin{align*}
F_{ (D_V, D_\mathfrak{g})}  (x) =~& D_V (F(x)) - F (D_\mathfrak{g} (x)) \\
=~& D_V ( (Us - sN) (x)) - (Us -sN) (D_\mathfrak{g} (x))\\
=~& (DUs - Us D_\mathfrak{g}) (x) - (Ds N - sN D_\mathfrak{g})(x) \quad (\because ~ \! D_V = D \big|_V)\\
=~& S ((Ds - s D_\mathfrak{g}) (x)) - (Ds - s D_\mathfrak{g}) (N (x)) \quad (\because ~ \! UD = DU,~  U \big|_V = S \text{ and } ND_\mathfrak{g} = D_\mathfrak{g} N),
\end{align*}
for $x, y \in \mathfrak{g}$. This shows that the $2$-cocycle $(\chi_{ (D_V, D_\mathfrak{g})}, F_{ (D_V, D_\mathfrak{g})} )$ is cohomologous to the null $2$-cocycle $(0, 0)$ by the map $Ds - s D_\mathfrak{g} : \mathfrak{g} \rightarrow V$. Hence $\mathcal{W} (D_V, D_\mathfrak{g}) = [ (\chi_{ (D_V, D_\mathfrak{g})}, F_{ (D_V, D_\mathfrak{g})})] = [(0,0)] = 0$.

\medskip

Conversely, assume that $(D_V, D_\mathfrak{g}) \in \mathcal{D}_\rho$ and $\mathcal{W} (D_V, D_\mathfrak{g})  = 0$. Choose any section $s$ and let the abelian extension produce the $2$-cocycle $(\chi, F)$. Then it follows that the $2$-cocycle $(\chi_{ (D_V, D_\mathfrak{g})}, F_{ (D_V, D_\mathfrak{g})})$ is cohomologous to the null $2$-cocycle. Hence there exists a linear map $\varphi : \mathfrak{g} \rightarrow V$ such that
%\begin{align*}
%chi_{(D_V, D_\mathfrak{g})} (x, y) = \rho_x \varphi(y)- \rho_y \varphi (x) - \varphi ([x, y]_\mathfrak{g}) \quad \text{ and } \quad F_{(D_V, D_\mathfrak{g})} (x) = S (\varphi (x)) - \varphi (N(x)),
%\end{align*}
\begin{align}
D_V (\chi (x, y)) - \chi (D_\mathfrak{g} (x), y) - \chi (x, D_\mathfrak{g} (y))  =~& \rho_x \varphi (y) - \rho_y \varphi (x) - \varphi ([x, y]_\mathfrak{g}) \label{first-use}\\
\text{ and } ~~ D_V (F(x)) - F (D_\mathfrak{g} (x)) =~& S (\varphi (x)) - \varphi (N(x)), \label{sec-use}
\end{align}
for $x, y \in \mathfrak{g}$. Using the fact that any element $e \in \mathfrak{e}$ can be uniquely written as $e = s(x) + u$ (for some $x \in \mathfrak{g}$ and $u \in V$), we now define a linear map $D : \mathfrak{e} \rightarrow \mathfrak{e}$ by
\begin{align*}
D(e) = D (s (x) + u) = s (D_\mathfrak{g} (x)) + D_V (u) + \varphi (x), \text{ for } e = s(x) + u \in \mathfrak{e}.
\end{align*}
Then for any $e = s(x) + u$ and $e' = s(y) + v$ from the space $\mathfrak{e}$, we have
\begin{align*}
D ([e, e']_\mathfrak{e}) =~& D ( [s(x) + u , s(y) + v]_\mathfrak{e})\\
=~& D ( [s(x) , s(y)]_\mathfrak{e} + [s(x), v]_\mathfrak{e} + [u, s(y)]_\mathfrak{e})\\
=~& D \big(  s [x, y]_\mathfrak{g} + \chi (x, y) + \rho_x v - \rho_y u  \big)\\
=~& s (D_\mathfrak{g} [x, y]_\mathfrak{g}) + D_V (\chi (x, y) + \rho_x v - \rho_y u) + \varphi ([x, y]_\mathfrak{g}) \qquad (\text{using definition of } D)\\
=~& s ( [D_\mathfrak{g} (x), y]_\mathfrak{g} + [x, D_\mathfrak{g} (y)]_\mathfrak{g}) + \chi (D_\mathfrak{g} (x), y ) + \chi (x, D_\mathfrak{g} (y)) + \rho_x \varphi (y) - \rho_y \varphi (x) \\
& \qquad + \rho_{D_\mathfrak{g} (x)} v + \rho_x D_V (v) - \rho_{D_\mathfrak{g} (y)} u - \rho_y D_V (u) \qquad (\text{by }  (\ref{ind-der}) \text{ and } (\ref{first-use}))\\
=~& [s (D_\mathfrak{g} (x)), s(y)]_\mathfrak{e} + [s(x), s (D_\mathfrak{g} (y)]_\mathfrak{e} + [s(x), \varphi (y)]_\mathfrak{e} + [\varphi (x), s(y)]_\mathfrak{e}  \\
& \qquad + [s (D_\mathfrak{g} (x)), v]_\mathfrak{e} + [s(x), D_V (v)]_\mathfrak{e} + [u, s(D_\mathfrak{g} (y))]_\mathfrak{e} + [D_V (u) , s(y)]_\mathfrak{e} \\
=~&  [ s (D_\mathfrak{g} (x)) + D_V (u) + \varphi (x) ~ \! , ~ \! s(y)+ v]_\mathfrak{e} + [ s(x) + u ~ \! , ~ \!   s (D_\mathfrak{g} (y)) + D_V (v) + \varphi (y) ]_\mathfrak{e}  \\
=~& [D(e), e']_\mathfrak{e} + [e, D (e')]_\mathfrak{e}.
\end{align*}
Moreover, similar to (\ref{compos}), one can show that $U \circ D = D \circ U$. Hence $D \in \mathrm{Der} (\mathfrak{e}, U)$ is a Nijenhuis Lie algebra derivation. Further, from the definition of $D$, we get that $D (V) \subset V$ which implies that $D \in \mathrm{Der}_V (\mathfrak{e}, U)$. Additionally, it is easy to see that $D \big|_V = D_V$ and $\overline{D} = pDs = D_\mathfrak{g}$. Hence $\eta (D) = (D \big|_V, \overline{D}) = (D_V, D_\mathfrak{g})$ which shows that the pair $(D_V, D_\mathfrak{g})$ is inducible.
\end{proof}

In the following, we construct the Wells exact sequence for Nijenhuis Lie algebra derivations in a given abelian extension.

\begin{thm}\label{thm-wells-ses-der}
Let (\ref{abel-last}) be an abelian extension of the Nijenhuis Lie algebra $(\mathfrak{g}, [~,~]_\mathfrak{g}, N)$ by the Nijenhuis representation $(V,\rho, S)$. Then there is a short exact sequence
\begin{align}\label{wells-ses-der-abel}
0 \rightarrow \mathrm{Der} ((\mathfrak{g}, N); (V, S)) \xrightarrow{\iota} \mathrm{Der}_V (\mathfrak{e}, U) \xrightarrow{\eta} \mathcal{D}_\rho \xrightarrow{\mathcal{W}} H^2 ((\mathfrak{g}, N); (V, S)),
\end{align}
where the map $\iota$ is given by $\iota (d) (e) = d (p(e))$, for $d \in \mathrm{Der} ((\mathfrak{g}, N); (V, S))$ and $e \in \mathfrak{e}$.
\end{thm}

\begin{proof}
It is easy to see that the map $\iota$ is injective. Hence the above sequence is exact in the first place. To show that it is exact in the second place, we take an element $D \in \mathrm{ker} (\eta)$. That is, $\eta (D) = (D \big|_V, \overline{D}) = (0,0)$ which means that $D \big|_V = 0$ and $\overline{D} = p Ds = 0$, where $s$ is any arbitrary but fixed section. For any $x \in \mathfrak{g}$, since $(pDs)(x) = 0$, it follows that $D (s(x)) \in V$. We define a map $d : \mathfrak{g} \rightarrow V$ by $d(x) := D (s(x)),$ for $x \in \mathfrak{g}$. Using the fact that $D \in \mathrm{Der}_V (\mathfrak{e}, U)$, it can be easily checked that $d \in \mathrm{Der} ((\mathfrak{g}, N); (V, S))$. Moreover, for any $e = s (x) + u \in \mathfrak{e}$, we have
\begin{align*}
D (e) = D (s (x) + u) = D (s (x)) = d(x) = \iota (d) (s (x) + u) = \iota (d) (e).
\end{align*}
This shows that $D \in \mathrm{im} (\iota)$ and hence $\mathrm{ker} (\eta) \subset \mathrm{im} (\iota)$. Conversely, for any $d \in \mathrm{Der} ((\mathfrak{g}, N); (V, S))$, we have $\eta ( \iota (d)) = \big( \iota (d) \big|_V ~ \! , ~ \! p (\iota (d)) s \big) = (0, 0)$ from the definition of $\iota (d)$. Hence $\mathrm{im} (\iota) \subset \mathrm{ker} (\eta)$ which in turn implies that $\mathrm{ker} (\eta) = \mathrm{im} (\iota)$. Thus, the sequence is exact in the second place. Finally, we take an element $(D_V, D_\mathfrak{g}) \in \mathrm{ker} (\mathcal{W}).$ Then by Theorem \ref{thm-ind-der}, the pair $(D_V, D_\mathfrak{g})$ is inducible. Hence it lies in the image of $\eta$. On the other hand, if $(D_V, D_\mathfrak{g}) \in \mathrm{im} (\eta)$ then it is inducible and hence $\mathcal{W} (D_V, D_\mathfrak{g}) = 0$. Therefore, $\mathrm{ker} (\mathcal{W}) = \mathrm{im} (\iota)$ which shows that the sequence is also exact in the third place. 
\end{proof}

When (\ref{abel-last}) is a split abelian extension of the Nijenhuis Lie algebra $(\mathfrak{g}, [~,~]_\mathfrak{g}, N)$ by the Nijenhuis representation $(V, \rho, S)$, the Wells map $\mathcal{W}: \mathcal{D}_\rho \rightarrow H^2 ((\mathfrak{g}, N); (V, S))$ vanishes identically. Hence the Wells exact sequence (\ref{wells-ses-der-abel}) becomes
\begin{align}\label{takes-the-form-der}
0 \rightarrow \mathrm{Der} ((\mathfrak{g}, N); (V, S)) \xrightarrow{\iota} \mathrm{Der}_V (\mathfrak{e}, U) \xrightarrow{\eta} \mathcal{D}_\rho \rightarrow 0.
\end{align}
For any $(D_V, D_\mathfrak{g}) \in \mathcal{D}_\rho$, we define a map $D_{(D_V, D_\mathfrak{g})} : \mathfrak{e} \rightarrow \mathfrak{e}$ by $D_{(D_V, D_\mathfrak{g})} (s(x) + u) := s (D_\mathfrak{g} (x)) + D_V (u)$, for any $s(x) + u \in \mathfrak{e}$. It is easy to verify that $D_{(D_V, D_\mathfrak{g})} \in \mathrm{Der}_V (\mathfrak{e}, U)$. Hence there is a map $t : \mathcal{D}_\rho \rightarrow \mathrm{Der}_V (\mathfrak{e}, U)$ given by $t (D_V, D_\mathfrak{g}) := D_{(D_V, D_\mathfrak{g})}$, for $(D_V, D_\mathfrak{g}) \in \mathcal{D}_\rho$. Further, one can show that $t$ is a Lie algebra homomorphism and also $\eta t = \mathrm{Id}_{\mathcal{D}_\rho}$. As a result, (\ref{takes-the-form-der}) is a split exact sequence of Lie algebras. Hence, in this case, we obtain the following isomorphism as Lie algebras
\begin{align*}
\mathrm{Der}_V (\mathfrak{e}, U) ~ \! \cong ~ \! \mathcal{D}_\rho \ltimes \mathrm{Der} ((\mathfrak{g}, N); (V, S)).
\end{align*}

\medskip

\noindent {\bf Future works.} In this paper, we develop the non-abelian cohomology theory of Nijenhuis Lie algebras and find applications to the inducibility problems. In particular, we define the second cohomology group of a Nijenhuis Lie algebra with coefficients in a Nijenhuis representation. We hope that Nijenhuis Lie algebra derivations will play a crucial role in defining the first cohomology group. However, the full cochain complex of a Nijenhuis Lie algebra is yet to be discovered. In subsequent work, we will discover the full cochain complex of a Nijenhuis Lie algebra with coefficients in a Nijenhuis representation. We also look forward to finding applications in deformation theory and homotopy algebras.

\medskip

Given a type of algebraic structure, it is often useful to construct a (differential) graded Lie algebra or more generally an $L_\infty$-algebra whose Maurer-Cartan elements correspond to a prescribed type of structure. In \cite{fre} Fr\'{e}gier constructed a differential graded Lie algebra whose Maurer-Cartan elements are precisely non-abelian extensions of a Lie algebra $\mathfrak{g}$ by another Lie algebra $\mathfrak{h}$. The construction of this differential graded Lie algebra is based on the Nijenhuis-Richardson graded Lie bracket that characterizes Lie algebra structures on a given vector space. The result of \cite{fre} was also generalized to the context of associative algebras and Leibniz algebras \cite{gouray}, \cite{liu-sheng-wang}. In future, we aim to construct an $L_\infty$-algebra whose Maurer-Cartan elements are precisely Nijenhuis Lie algebra structures on a given vector space. This characterization could be useful in developing the cohomology theory of Nijenhuis Lie algebras as well as in obtaining the Maurer-Cartan characterization of non-abelian extensions of Nijenhuis Lie algebras.

\medskip

It has been observed in \cite{das-twisted} that Nijenhuis operators on Lie algebras are closely related to twisted Rota-Baxter operators and NS-Lie algebras (a generalization of pre-Lie algebras). A class of twisted Rota-Baxter operators is given by Reynolds operators on Lie algebras. An interesting study could be to understand the non-abelian extension theory of Reynolds Lie algebras and NS-Lie algebras with possible connections to the non-abelian extensions of Rota-Baxter Lie algebras and Nijenhuis Lie algebras.

\medskip

\noindent {\bf Acknowledgements.} The author thanks the Department of Mathematics, IIT Kharagpur for providing the beautiful academic atmosphere where the research has been carried out.

\medskip

\noindent {\bf Data Availability Statement.} Data sharing does not apply to this article as no new data were created or analyzed in this study.


\begin{thebibliography}{BFGM03}


\bibitem{azimi} M. J. Azimi, C. Laurent-Gengoux and J. M. Nunes da Costa, Nijenhuis forms on $L_\infty$-algebras and Poisson geometry, {\em Diff. Geom. Appl.} 38 (2015), 69-113.

%\bibitem{aamar} M. Ammar and N. Poncin, Coalgebraic approach to the Loday infinity category, stem differential for $2n$-ary graded and homotopy algebras, {\em Annales de l'institut Fourier, Tome} 60, no. 1 (2010) 355-387.

%\bibitem{agore} A. L. Agore, Classifying complements for associative algebras, {\em Linear Alg. Appl.} 446 (2014), 345-355.

\bibitem{agore} A. L. Agore and G. Militaru, Extending structures for Lie algebras, {\em Monatsh. fur Mathematik} 174 (2014), 169-193.

%\bibitem{agore3} A. L. Agore and G. Militaru, Classifying complements for Hopf algebras and Lie algebras, {\em J. Algebra} 391 (2013), 321-341.

%\bibitem{agore2} A. L. Agore and G. Militaru, The extending structures problem for algebras, Available at: \url{https://arxiv.org/abs/1305.6022v3}

%\bibitem{agore} A. L. Agore and G. Militaru, Unified products for Leibniz algebras. Applications, {\em Linear Algebra Appl.} 439 (2013), 2609-2633.

%\bibitem{agore4} A. L. Agore and G. Militaru, Hochschild products and global non-abelian cohomology for algebras. Applications, {\em J. Pure Applied Algebra} 221 (2017), 366-392.

%\bibitem{aguiar} M. Aguiar, Pre-Poisson algebras, {\em Lett. Math. Phys.} 54 (2000), 263-277.

%\bibitem{Bai-Guo-Ni} C. Bai, L. Guo, and X. Ni, $\mathcal{O}$-operators on associative algebras and associative Yang–Baxter equations, {\em Pacific J. Math.} 256 (2012), 257-289.

%\bibitem{guo-commun} C. Bai, L. Guo and X. Ni, Nonabelian generalized Lax pairs, the classical Yang-Baxter equation and PostLie algebras, {\em Commun. Math. Phys.} 297 (2010), 553-596.

%\bibitem{Bai-Guo-Ni} C. Bai, L. Guo, and X. Ni, Relative Rota-Baxter algebras and tridendriform algebras, {\em J. Algebra Appl.} 12, No. 07 (2013), 1350027.

%\bibitem{baishya-das} A. Baishya and A. Das, Cup product, Fr\"{o}licher-Nijenhuis bracket and the derived bracket associated to Hom-Lie algebras, arXiv:2409.01865

%\bibitem{bala} D. Balavoine, Deformations of algebras over a quadratic operad, {\em Contemp. Math.} 202 (1967), 207-234.

\bibitem{barati} M. Barati and F. Saeedi, Derivations of abelian Lie algebra extensions, {\em Note Mat.} 39 (2019), 71-86.

\bibitem{bar-singh} V. G. Bardakov and M. Singh, Extensions and automorphisms of Lie algebras, {\em J. Algebra Appl.} 16 (2017) 1750162, 15pp. 

%\bibitem{bloh} A. Bloh, On a generalization of the concept of Lie algebra, {\em Dokl. Akad. Nauk. SSSR.} 165(3) (1965), 471-473.

%\bibitem{bor} M. Bordemann and F. Wagemann, Global integration of Leibniz algebras, {\em J. Lie Theory} 27 (2017), no. 2, 555–567.

%\bibitem{cao} W. Cao, An algebraic study of averaging operators, Ph.D. Thesis, Rutgers University at Newark (2000).

\bibitem{gra-bi} J. F. Cari\~{n}ena, J. Grabowski and G. Marmo, Quantum Bi-Hamiltonian systems, {\em Int. J. Mod. Phys. A} 15 (2000), 4797-4810. 

\bibitem{casas} J. M. Casas, E. Khmaladze and M. Ladra, Low-dimensional non-abelian Leibniz cohomology,
{\em Forum Math.} 25 (2013), no. 3, 443-469.

%\bibitem{conn} A. Connes and D. Kreimer, Renormalization in quantum field theory and the Riemann-Hilbert problem I: the Hopf algebra structure of graphs and the main theorem, {\em Commun. Math. Phys.} 210 (2000), 249-273.

%\bibitem{covez} S. Covez, The local integration of Leibniz algebras, {\em Ann. Inst. Fourier (Grenoble)} 63 (2013), no. 1, 1-35.

%\bibitem{das-rota} A. Das, Deformations of associative Rota-Baxter operators, {\em J. Algebra} 560 (2020) 144-180.

\bibitem{das-twisted} A. Das,  Twisted Rota-Baxter operators and Reynolds operators on Lie algebras and NS-Lie algebras, {\em J. Math. Phys.} Vol. 62, Issue 9  (2021) 091701.

%\bibitem{das-leib-hrs} A. Das, Leibniz algebras with derivations, {\em J. Homotopy Relat. Struct.} 16 (2021) 245-274.

%\bibitem{das-ns} A. Das, Cohomology and deformations of twisted Rota-Baxter operators and NS-algebras, {\em  J. Homotopy Relat. Struct.} 17 (2022) 233-262.

%\bibitem{das-weighted} A. Das, Cohomology and deformations of weighted Rota-Baxter operators, {\em J. Math. Phys.} Vol. 63, Issue 9 (2022) 091703.

%\bibitem{das-crossed} A. Das, Cohomology and deformations of crossed homomorphisms, {\em Bull. Belgian Math. Soc. Simon Stevin} Vol. 28, Issue 3 (2022) 381-397.

%\bibitem{das-modified} A, Das, A cohomological study of modified Rota-Baxter algebras, Available at: \url{https://arxiv.org/abs/2207.02273}

%\bibitem{das-avg} A. Das, Controlling structures, deformations and homotopy theory for averaging algebras, Available at: \url{https://arxiv.org/abs/2303.17798}

%\bibitem{das-leib} A. Das, Relative Rota-Baxter Leibniz algebras, their characterization and cohomology, {\em Linear Multilinear Algebra} Vol. 71, Issue 17 (2023) 2796-2822.

%\bibitem{das-leib-pub} A. Das, Weighted relative Rota-Baxter operators on Leibniz algebras and Post-Leibniz algebra structures, {\em Publ. Math. Debrecen} 103, Issue 3-4 (2023) 385-406.

%\bibitem{das-guo} A. Das and S. Guo, Cohomology and deformations of generalized Reynolds operators on Leibniz algebras, {\em Rocky Mountain J. Math.} 54 (1) (2024), 161-178.

%\bibitem{das-mandal} A. Das and R. Mandal, Quasi-twilled associative algebras, deformation maps and their governing algebras, Available at: \url{https://arxiv.org/abs/2409.00443}

%\bibitem{das-mishra} A. Das and S. K. Mishra, The $L_\infty$-deformations of associative Rota-Baxter algebras and homotopy Rota-Baxter operators, {\em J. Math. Phys.} Vol. 63, Issue 5 (2022) 051703.

\bibitem{das-rathee} A. Das and N. Rathee, Extensions and automorphisms of Rota-Baxter groups, {\em J. Algebra} 636 (2023), 626-665.

%\bibitem{das-sen} A. Das and S. Sen, Nijenhuis operators on Hom-Lie algebras, {\em Comm. Algebra} 50, Issue 3 (2022), 1038-1054.

\bibitem{ded} P. Dedecker and A. S.-T. Lue, A non-abelian two-dimensional cohomology for associative algebras, {\em Bull. Amer. Math. Soc.} 72 (1966), 1044-1050.

%\bibitem{dherin} B. Dherin and F. Wagemann, Deformation quantization of Leibniz algebras, {\em Adv. Math.} 270 (2015), 21–48.

\bibitem{dorfman} I. Dorfman, Dirac structures and integrability of nonlinear evolution equations, Wiley, 1993.

%\bibitem{drinfeld} V. Drinfeld, Quasi-Hopf algebras, {\em Leningrad Math. J.} 1 (1989), 1419-1457.

%\bibitem{ebrahimi} K. Ebrahimi-Fard, Loday-type algebras and the Rota-Baxter relation, {\em Lett. Math. Phys.} 61, no. 2 (2002), 139-147.

\bibitem{eilen} S. Eilenberg and S. Maclane, Cohomology theory in abstract groups. II. Group extensions with non-abelian kernel, {\em Ann. of Math.} 48 (1947) 326-341. 

\bibitem{fre} Y. Fr\'{e}gier, Non-abelian cohomology of extensions of Lie algebras as Deligne groupoid, {\em J. Algebra} 398 (2014) 243-257.

%\bibitem{yau} Y. Fr\'{e}gier, M. Markl and D. Yau, The $L_\infty$-deformation complex of diagrams of algebras, {\em New York J. Math.} 15 (2009), 353-392.

%\bibitem{fre} Y. Fr\'{e}gier and M. Zambon, Simultaneous deformations of algebras and morphisms via derived brackets, {\em J. Pure Appl. Algebra} 219 (2015), 5344-5362.

\bibitem{frol} A. Fr\"{o}licher and A. Nijenhuis, Theory of vector valued differential forms. Part I. {\em Indag. Math.}, 18 (1956), 338-360.

%\bibitem{gers-ring} M. Gerstenhaber, The cohomology structure of an associative ring, {\em Ann. Math. (2)} 78 (1963), 267-288.

%\bibitem{gers}  M. Gerstenhaber, On the deformation of rings and algebras, {\em Ann. Math. (2)} 79 (1964), 59-103.

%\bibitem{gers-sch} M. Gerstenhaber and S. D. Schack, On the deformation of algebra morphisms and diagrams, {\em Trans. Amer. Math. Soc.} 279 (1983), no. 1, 1-50.

%\bibitem{getzler} E. Getzler, Lie theory for nilpotent $L_{\infty}$-algebras, {\em Ann. Math. (2)} 170 (2009), 271-301.

\bibitem{gouray} J.-B. Gouray, A differential graded Lie algebra approach to non-abelian extensions of associative algebras, Online available at: \url{https://arxiv.org/abs/1802.04641}


%\bibitem{guo-book} L. Guo, An introduction to Rota-Baxter algebra, {\em International Press, Somerville, MA; Higher Education Press, Beijing}, 2012. xii+226 pp. ISBN:978-1-57146-253-4

%\bibitem{guo-hou} Y. Guo and B. Hou, Crossed modules and non-abelian extensions of Rota-Baxter Leibniz algebras, {\em J. Geom. Phys.} 191 (2023), 104906.

\bibitem{hazra-habib} S. K. Hazra and A. Habib, Wells exact sequence and automorphisms of extensions of Lie superalgebras, {\em J. Lie Theory} 30 (2020) 179-199.

%\bibitem{hoch} G. Hochschild, On the cohomology groups of an associative algebra, {\em Ann. Math. (2)} 46 (1945), 58-67.

\bibitem{hochschild} G. Hochschild and J.-P. Serre, Cohomology of group extensions, {\em Trans. Amer. Math. Soc.} 74 (1953) 110-134.

\bibitem{hoch} G. Hochschild, Cohomology classes of finite type and finite-dimensional kernels for Lie algebras, {\em Am. J. Math.} 76 (1954) 763-778.

\bibitem{hou-zhao} B. Hou and J. Zhao, Crossed modules, non-abelian extensions of associative conformal algebras and Wells exact sequences, {\em J. Algebra Appl.} to appear. DOI: https://doi.org/10.1142/S0219498826500337

\bibitem{inas} N. Inassaridze, E. Khmaladze and M. Ladra, Non-abelian cohomology and extensions of Lie algebras, {\em J. Lie Theory} 18 (2008) 413-432.

\bibitem{jiang-sheng} J. Jiang and Y. Sheng, Representations and cohomologies of relative Rota-Baxter Lie algebras and applications, {\em J. Algebra} 602 (2022), 637-670.

%\bibitem{jiang} J. Jiang and Y. Sheng, Deformations, cohomologies and integrations of relative difference Lie algebras, {\em J. Algerba} 614 (2023), 535-563.

%\bibitem{jiang-sheng-tang} J. Jiang, Y. Sheng and R. Tang, Deformation maps of quasi-twilled Lie algebras, Available at: \url{https://arxiv.org/abs/2405.02532}

\bibitem{jin} P. Jin and H. Liu, The Wells exact sequence for the automorphism group of a group extension, {\em J. Algebra}
324 (2010) 1219-1228.

%\bibitem{kaj-stas} H. Kajiura and J. Stasheff, Homotopy algebras inspired by classical open-closed string field theory, {\em Commun. Math. Phys.} 263 (2006), 553-581.

%\bibitem{khuda} D. Khudaverdyan, N. Poncin and J. Qiu, On the infinity category of homotopy Leibniz algebras, {\em Theory Appl. Categ.} Vol. 29, No. 12 (2014), pp 332-370.

\bibitem{koss} Y. Kosmann-Schwarzbach and F. Magri, Poisson-Nijenhuis structures,  {\em Annales de l'I.H.P. Physique th\'{e}orique} 53, no. 1 (1990), 35-81.

%\bibitem{kotov} A. Kotov and T. Strobl, The embedding tensor, Leibniz-Loday algebras, and their higher gauge theories, {\em Commun. Math. Phys.} 376 (2020), no. 1, 235–258.

%\bibitem{jonas} A. Kraft and J. Schnitzer, An introduction to $L_\infty$-algebras and their homotopy theory for the working mathematician,  {\em Rev. Math. Phys.} 36, No. 01 (2024), 2330006.

%\bibitem{kuper} B. A. Kupershmidt, What a classical {\em r}-matrix really is, {\em J. Nonlinear Math. Phys.} 6, No. 4 (1998), 448-488.

%\bibitem{lada-markl} T. Lada and M. Markl, Strongly homotopy Lie algebras, {\it Comm. Algebra} 23 (1995), 2147-2161.

\bibitem{lei} P. Lei and L. Guo, Nijenhuis algebras, NS algebras, and N-dendriform algebras, {\em Front. Math. China} 7 (2012), 827–846.

\bibitem{leroux} P. Leroux, Construction of Nijenhuis operators and dendriform trialgebras, {\em Int. J. Math. Math. Sci.} 49 (2004), 2595-2615.

\bibitem{liu-sheng} J. Liu, Y. Sheng, Y. Zhou and C. Bai, Nijenhuis operators on $n$-Lie algebras, {\em Commun. Theor. Phys.} 65 (2016), pp. 659.

\bibitem{liu-sheng-wang} J. Liu, Y. Sheng and Q. Wang, On non-abelian extensions of Leibniz algebras, {\em Comm. Algebra} 46 (2018), 574-587.

%\bibitem{loday-dialgebra} J.-L. Loday, Dialgebra, in Dialgebras and related operad, {\em Lecture Notes in Math.} 1763 (2002), 7-66.

%\bibitem{lazarev} A. Lazarev, Y. Sheng and R. Tang, Deformations and homotopy theory of relative Rota-Baxter Lie algebras, {\em Commun. Math. Phys.} 383 (2021), 595-631.

%\bibitem{li-wang} Y. Li and D. Wang, Cohomology and deformation theory of crossed homomorphisms of Leibniz algebras, {\em J. Algebra Appl.} to appear. DOI: https://doi.org/10.1142/S0219498825501956

%\bibitem{Liu-Bai-sheng} J. Liu, C. Bai and Y. Sheng, Compatible $\mathcal{O}$-operators on bimodules over associative algebras, {\em J. Algebra }532 (2019), 80-118.

%\bibitem{loday-ronco} J.-L. Loday and M. Ronco, Trialgebras and families of polytopes, {\em Contemp. Math.} 346 (2004), 369-398.

%\bibitem{loday-morph} J.-L. Loday,  On the operad of associative algebras with derivation, {\em Georgian Math. J.} 17 (2010), no. 2, 347-372.

%\bibitem{lod-val-book} J.-L. Loday and B. Vallette, Algebraic operads, Grundlehren mathematischen Wissenschaften, Volume 346, Springer-Verlag (2012), xviii+512 pp.


%\bibitem{gers}  M. Gerstenhaber, On the deformation of rings and algebras, {\em Ann. of Math. (2)} 79 (1964), 59-103.

%\bibitem{loday} J.-L. Loday, Une version non commutative des alg\'{e}bres de Lie: les alg\'{e}bres de Leibniz, {\em Enseign. Math.} (2) 39 (1993), no. 3-4, 269-293.

%\bibitem{loday-pira}  J.-L. Loday and T. Pirashvili, Universal enveloping algebras of Leibniz algebras and (co)homology, {\em Math. Ann.} 296 (1993), 139-158.

\bibitem{lue} A. S.-T. Lue, Non-abelian cohomology of associative algebras, {\em Quart. J. Math.} 19 (1968), 159-180.

\bibitem{ma} T. Ma and L. Long, Nijenhuis operators and associative $D$-bialgebras, {\em J. Algebra} 639 (2024), 150-186. 

%\bibitem{ma-song-wang} Q. Ma, L. Song and Y. Wang, Non-abelian extensions of pre-Lie algebras, {\em Comm. Algebra} 51 (2023), 1370-1382.

%\bibitem{mandal} A. Mandal, Deformations of Leibniz algebra morphisms, {\em Homology Homotopy Appl.} 9(2007), no. 1, 439–450.

%\bibitem{men} I. Mencattini and A. Quesney, Crossed homomorphisms, integration of Post-Lie algebras and the Post-Lie magnus expansion, {\em Comm. Algebra} 49 (2021), 3507-3533.

\bibitem{das-hazra-mishra} S. K. Mishra, A. Das and S. K. Hazra,
Non-abelian extensions of Rota-Baxter Lie algebras and inducibility of automorphisms, {\em Linear Alg. Appl.} 669 (2023), 147-174.

%\bibitem{saha} B. Mondal and R. Saha, Cohomology of modified Rota-Baxter Leibniz algebra of weight $\lambda$, {\em J. Algebra Appl.} to appear. DOI: https://doi.org/10.1142/S0219498825501579

\bibitem{saha} B. Mondal and R. Saha, Nijenhuis operators on Leibniz algebras, {\em J. Geom. Phys.} 196 (2024), 105057.

%\bibitem{nij-ric} A. Nijenhuis and R. Richardson, Cohomology and deformations in graded Lie algebras, {\em Bull. Amer. Math. Soc.} 72 (1966), 1-29.

%\bibitem{nij-ric2} A. Nijenhuis and R. Richardson, Deformations of homomorphisms of Lie groups and Lie algebras, {\em Bull. Amer. Math. Soc.} 73 (1967), 175-179.

%\bibitem{pei} Y. Pei, Y. Sheng, R. Tang and K. Zhao, Actions of monoidal categories and representations of Cartan type Lie algebras, {\em J. Inst. Math. Jussieu} 22 (2023), 2367-2402.

\bibitem{peng-zhang} X.-S. Peng and Y. Zhang, Extending structures of Rota-Baxter Lie algebras, Arxiv preprint, Online available at: \url{https://arxiv.org/abs/2306.15874}.

%\bibitem{saha} B. Mondal and R. Saha, Nijenhuis operators on Leibniz algebras, {\em J. Geom. Phys.} 196 (2024), 105057.

%\bibitem{sheng} Y. Sheng, A survey on deformations, cohomologies and homotopies of relative Rota-Baxter Lie algebras, {\em Bull. London Math. Soc.} 54 (2022), 2045-2065.

%\bibitem{sheng-embed} Y. Sheng, R. Tang and C. Zhu, The controlling $L_\infty$-algebra, cohomology and homotopy of embedding tensors and Lie-Leibniz triples, {\em Commun. Math. Phys.} 386 (2021), 269-304.

\bibitem{tan-xu} Y. Tan and S. Xu, On a Lie algebraic approach to abelian extensions of associative algebras, {\em Canad. Math. Bulletin} 64 (2021), 25-38.

%\bibitem{tang} R. Tang, C. Bai, L. Guo and Y. Sheng, Deformations and their controlling cohomologies of $\mathcal{O}$-operators, {\em Commun. Math. Phys.} 368 (2019), 665-700.

%\bibitem{tang-sheng} R. Tang and Y. Sheng, Leibniz bialgebras, relative Rota–Baxter operators, and the classical Leibniz Yang–Baxter equation, {\em J. Noncommut. Geom.} 16 (2022), 1179–1211.

%\bibitem{tang-sheng-zhou} R. Tang, Y. Sheng and Y. Zhou, Deformations of relative Rota–Baxter operators on Leibniz algebras, {\em Int. J. Geom. Methods Mod. Phys.} Vol. 17, No. 12 (2020), 2050174.

%\bibitem{tang-xu-sheng} R. Tang, N. Xu and Y. Sheng, Symplectic structures, product structures and complex structures on Leibniz algebras, {\em J. Algebra} 647 (2024), 710-743.

%\bibitem{uchino}  K. Uchino, Quantum analogy of Poisson geometry, related dendriform algebras and Rota-Baxter operators, {\em Lett. Math. Phys.} 85, no. 2-3 (2008), 91-109.

%\bibitem{uchino-t}  K. Uchino, Twisting on associative algebras and Rota–Baxter type operators, {\em J. Noncommut. Geom.} 4 (2010), 349–379.

%\bibitem{voro} Th. Voronov, Higher derived brackets and homotopy algebras, {\em J. Pure Appl. Algebra} 202 (2005), 133-153.

\bibitem{wang} Q. Wang, Y. Sheng, C. Bai and J. Liu, Nijenhuis operators on pre-Lie algebras, {\em Commun. Contemp. Math.} 21, No. 07 (2019), 1850050.

%\bibitem{wang-zhou} K. Wang and G. Zhou, Cohomology theory of averaging algebras, $L_\infty$-structures and homotopy averaging algebras, Available at: \url{https://arxiv.org/abs/2009.11618}

%\bibitem{wang-zhou2} K. Wang and G. Zhou, Deformations and homotopy theory of Rota-Baxter algebras of any weight, Available at: \url{https://arxiv.org/abs/2108.06744}

\bibitem{wells} C. Wells, Automorphisms of group extensions, {\em Trans. Amer. Math. Soc.} 155 (1971) 189-194.

\bibitem{yuan} L. Yuan, $\mathcal{O}$-operators and Nijenhuis operators of associative conformal algebras, {\em J. Algebra} 609 (2022), 245-291.

%\bibitem{zhang-gao-guo} T. Zhang, X. Gao and L. Guo, Reynolds algebras and their free objects from bracketed words and rooted trees, {\em J. Pure Applied Algebra} 225 (2021) 106766.

\end{thebibliography}
\end{document}